\documentclass[a4paper]{amsart}
\usepackage{hyperref}
\usepackage{stmaryrd}
\usepackage{latexsym}
\usepackage{amsmath}
\usepackage{amsfonts}
\usepackage{amssymb}
\usepackage{amsthm}
\usepackage{bm}
\usepackage[english]{babel}
\usepackage{graphicx}
\usepackage{color}
\usepackage[margin=3cm]{geometry}
\usepackage{tikz}
\usepackage{braket}
\usepackage{cite}

\newtheorem{theorem}{Theorem}

\newtheorem{corollary}{Corollary}

\newtheorem{lemma}{Lemma}

\newtheorem{proposition}{Proposition}
\newtheorem{definition}{Definition}

\DeclareMathOperator{\Li}{Li}
\DeclareMathOperator{\LOG}{L}

\newcommand{\bigO}{\mathcal{O}}

\newcommand{\be}{\begin{equation}}
\newcommand{\ee}{\end{equation}}
\newcommand{\ef}[1]{\, #1}
\newcommand{\dx}[1] {\!\mathrm{d}{#1}\,}

\newcommand{\smfrac}[2]{\genfrac{}{}{0.25pt}{1}{#1}{#2}}

\newcommand{\cZ}{\mathcal{Z}}

\newcommand{\cM}{\mathcal{M}}
\newcommand{\cN}{\mathcal{N}}

\newcommand{\vsig}{{\bm \sigma}}
\newcommand{\barh}{\bar{h}}

\title[The number of optimal matchings for Euclidean Assignment on the line]{The number of optimal matchings\\in the Euclidean Assignment Problem on the line}

\author[S. Caracciolo]{Sergio Caracciolo}
\author[V. Erba]{Vittorio Erba}  
\address[S. Caracciolo and V. Erba]{Dipartimento di Fisica, University of Milan and INFN, via Celoria 16, 20133 Milan, Italy}
\author[A. Sportiello]{Andrea Sportiello}
\address[A. Sportiello]{ LIPN, and CNRS, Universit\'e Paris 13, Sorbonne Paris Cit\'e, 99 Av.~J.-B.~Cl\'ement, 93430 Villetaneuse, France}

\begin{document}
\date{\today}

\begin{abstract} 
    We consider the Random Euclidean Assignment Problem in dimension
    $d=1$, with linear cost function.  In this version of the problem,
    in general, there is a large degeneracy of the ground state,
    i.e.\ there are many different optimal matchings 
    (say, $\sim \exp(S_N)$ at size $N$).  We characterize all possible
    optimal matchings of a given instance of the problem, and we give
    a simple product formula for their number.  

    Then, we study the probability distribution of $S_N$ (the
    zero-temperature entropy of the model), in the uniform random
    ensemble.  We find that, for large $N$, 
    $S_N \sim \smfrac{1}{2} N \log N + N s + \bigO\left( \log N
    \right)$, where $s$ is a random variable whose distribution $p(s)$
    does not depend on $N$.  We give expressions for the asymptotics
    of the moments of $p(s)$, both from a formulation as a Brownian process, and via singularity analysis of the generating
    functions associated to $S_N$.  The latter approach provides a
    combinatorial framework that allows to compute an asymptotic
    expansion to arbitrary order in $1/N$ for the mean and the variance of
    $S_N$.
\end{abstract}
\maketitle

\section{Introduction}

\noindent
The Euclidean Assignment Problem (EAP) is a combinatorial optimization
problem in which one has to pair $N$ white points to $N$ black points
minimizing a total cost function that depends on the Euclidean distances
among the points. The Euclidean Random Assignment Problem (ERAP)
is the statistical system in which these $2N$ points are
drawn from a given probability measure.
In the latter version of the problem, one is interested in characterizing
the statistical properties of the optimal solution, such as its average cost,
average structural properties, etc\dots

The EAP problem has applications in Computer Science, where it has been
used in computer graphics, image processing and machine
learning~\cite{peyre2020ComputationalOptimalTransport}, and it is also a
discretised version of a problem in functional analysis, called
\emph{optimal transport problem}, where the $N$-tuples of points are
replaced by probability measures.  The optimal transport problem has
recently seen a growing interest in pure mathematics, where it has
found applications in measure theory and gradient
flows~\cite{ambrosio2013UserGuideOptimala}.

To be more definite, the EAP problem is the optimization problem defined by
\begin{equation}
    \begin{split}
        \min_{\pi \in \mathfrak{S}_N} H_J(\pi)\, 
    \end{split}
\end{equation}
where
\begin{itemize}
    \item $J=(\{w_i\},\{b_j\})$ is the \emph{instance} of the
      problem, i.e., in the Euclidean version in dimension $d$, the
      datum of the positions of $N$ white points $\{w_i\}$ and $N$
      black points $\{b_j\}$ in $\mathbb{R}^d$;
    \item $\pi$ is a bijection between the two sets of points, linking
      biunivocally each white point to a unique black point. In other
      words, is a perfect matching. We denote by $\mathfrak{S}_N$ the
      set of all possible bijections between sets of size $N$;
    \item $H_J(\pi)$ is the cost function
        \begin{equation}
            \begin{split}
                H_J(\pi) = \sum_{i=1}^N c\left(\mathrm{dist}\left(w_i, b_{\pi(i)}\right)\right) \, 
            \end{split}
        \end{equation}
        i.e.\ the sum of the costs of the links of $\pi$, where a link
        is weighted using a \emph{link cost function}
        $c(x):\mathbb{R}^+\rightarrow\mathbb{R}$, depending only on
        the Euclidean distance among the two points.
\end{itemize}
In the EAP, $J$ is considered as fixed, while in the ERAP, $J$ is a
random variable with a fixed probability distribution.

It is a longstanding project to understand the phase diagram of the
model, in the plane $(p,d)$, when (in the simplest version of this
problem) one suppose that $J$ is given by $2N$ i.i.d.\ points from the
$d$-dimensional hypercube $[0,1]^d$, and that the link cost function
is given by $c(x)=x^p$.
This phase diagram is, to a certain extent, still mysterious, although
some progress has been made recently (see for example \cite{matteophdthesis}).  In particular, when $d=1$, the analytical
characterization of optimal solutions is quite
tractable~\cite{mccann1999ExactSolutionsTransportation,
  boniolo2014CorrelationFunctionGridPoisson}:
\begin{itemize}
    \item for $p>1$, the optimal matching is \emph{ordered}\footnote{
      The ordered matching is the one in which the first white point
      from the left is matched with the first black point and so on,
      and is represented by the identity permutation $\pi_{\rm ord}(i)
      = i$ if the points are sorted by increasing coordinate.  }
      independently on $J$, and this is due to the fact that the link
      cost function is increasing and convex.  In this case, plenty of
      results have been obtained on the statistics of the average
      optimal cost~\cite{Caracciolo:160, Caracciolo:169,
        Caracciolo:172, Caracciolo:177};
    \item for $0<p<1$, the optimal matching is \emph{non-crossing},
      meaning that pairs of matched points are either nested one
      inside the other, or disjoint (that is, matched pairs do not
      interlace) \cite{mccann1999ExactSolutionsTransportation}.  This
      property is imposed by the concavity of the link cost function.
      In this case, the optimal matching is not uniquely determined by
      the ordering of the points (although the viable candidates are
      typically reduced, roughly, from $n!$ to $\sqrt{n!}$), and a
      comparatively smaller number of results has been found so far
      for the random version of the
      problem~\cite{caracciolo2020DyckBoundConcave,
        bobkovTransportInequalitiesEuclidean, juillet2020solution,
        pAiry};
    \item for $p<0$, due to the fact that an overall positive factor
      in $c(x)$ is irrelevant in the determination of the optimal
      matching, it is questionable if one should consider the analytic
      continuation of the cost function $c(x) = x^p$ (which has the
      counter-intuitive property that the preferred links are the
      longest ones), or of the cost function $c(x) = p \, x^p$ (which
      has the property that the preferred links are the shortest ones,
      and that the limit $p \to 0$ is well-defined, as it corresponds
      to $c(x)=\log x$, but has the disadvantage of having average
      cost $-\infty$ when
        $p \leq -d$). In the first case, the optimal matching is
        \emph{cyclical}, meaning that the permutation that describes the
        optimal matching has a single cycle, and some result on the
        average optimal cost where obtained in \cite{Caracciolo:169}. In
        the second case, in the pertinent range $-d<p\leq 0$, it seems
        that the qualitative features of the $p \in \,(0,1)\,$ regime
        are preserved.
\end{itemize}
We notice that in all the cases above, as well as in all the cases
with $d \neq 1$ and any $p$, the optimal matching is `effectively
unique', meaning that it is almost surely unique, and, even in
presence of an instance showing degeneracy of ground states, almost
surely an infinitesimal perturbation immediately lifts the degeneracy.

This generic non-degeneracy property does not hold only for the
$d=p=1$ case (note that $p=1$ is the exponent at which the cost
function changes concavity). It is known
\cite{boniolo2014CorrelationFunctionGridPoisson} that in this case
there are (almost surely) at least two distinct optimal matchings for
each instance of the problem, the \emph{ordered matching} and the
\emph{Dyck matching} \cite{caracciolo2020DyckBoundConcave} 
which coincide only in the rather atypical case (with probability
$N!/(2N-1)!!$) in which, for all $i$, the $i$-th white and black
points are consecutive along the segment.


These observations suggest some fundamental questions for the $(d,p)=(1,1)$ problem:
\begin{enumerate}
    \item is it possible to characterize all the optimal matchings of a fixed instance $J$ of the EAP?
    \item how many optimal configurations are there for a fixed instance $J$ of the EAP?
    \item for random $J$'s, what are the statistical properties of the number of optimal configurations?
\end{enumerate}
The aim of this paper is to answer the three questions above.

Our interest is not purely combinatorial.  In fact, the ERAP is a
well-known and well-studied toy model of Euclidean spin glass
\cite{Orland1985,Mezard1985,Mezard1986a,Caracciolo:168}.  The
characterization of the set $\mathcal{Z}_J \subseteq \mathfrak{S}_N$
of the optimal matchings of $J$ is thus related to the computation of
the zero-temperature partition function $Z_J = |\mathcal{Z}_J|$ of the
disordered model, and of its zero-temperature entropy $S_J=\log(Z_J)$.

In this manuscript we will focus on the statistical properties of the
zero-temperature entropy, the thermodynamic potential that rules the
physics of the model when the disorder is \emph{quenched}, i.e.\ when
the timescale of the dynamics of the disorder degrees of freedom of
the model is much larger than that of the microscopic degrees of
freedom.  We leave to future investigations the study of the
\emph{annhealed} and \emph{replicated} partition functions $Z_J$ and
$Z_J^k$, which are less fundamental from the point of view of
statistical physics, but, as we will show elsewhere, have remarkable
combinatorial and number-theoretical properties.

\subsection{Summary of results}

In Section~\ref{sec:optimalMatch} we prove that, given a fixed
instance $J$ of the $(d,p) = (1,1)$ EAP problem, a matching $\pi$ is
optimal if and only if
\begin{equation}
    \begin{split}
        k_{\rm LB}(z) = k_\pi(z) \, 
    \end{split}
\end{equation}
where $k_{\rm LB}(z)$ is a function of $J$ that counts the difference
of the numbers of white and black points on the left of $z$, and
$k_\pi(z)$ counts the number of links of $\pi$ whose endpoints
lie on opposite sides of $z$.

In Section~\ref{sec:enum}, we show that $\cZ_J$, the set of optimal
matchings, depends on $J$ only through the ordering of the points,
while it is independent on their positions (provided that the ordering
is not changed).  Using a straightforward bijection between the
ordering of bi-colored point configurations on the line and a class of
lattice paths (the \emph{Dyck bridges}), we provide a combinatorial
recipe to construct the set $\cZ_J$.  As a corollary, we obtain a
rather simple product formula for the cardinality $Z_J := |\cZ_J|$,
that is, roughly speaking,
\begin{equation}
    \begin{split}
        Z_J = \prod_{\substack{ \rm descending \\ \rm steps}} \text{height of the step} \,
    \end{split}
\end{equation}
where the product runs over the descending steps of the Dyck bridge
associated with the ordering of the points of configuration $J$, and
the height of a step is, roughly speaking, the absolute value of its
vertical position.  This implies an analogous sum formula for the
entropy $S_J = \log(Z_J)$.

Then, we study the statistics of $S_J$ when $J$ is a random instance
of the problem.  Our techniques apply equally well, and with
calculations that can be performed in parallel, to two interesting
statistical ensembles:
\begin{description}
    \item[Dyck bridges] the case in which white and black points are just
        i.i.d.\ on the unit interval. 
    \item[Dyck excursions] the restriction of the previous ensemble to the
        case in which, for all $z \in [0,1]$, there are at least as many
        white points on the left of $z$ than black points.
\end{description}
The fact that we can study these two ensembles in parallel is present
also in our study for the distribution of the energy distribution of
the ``Dyck matching'', that we perform
elsewhere~\cite{caracciolo2020DyckBoundConcave,pAiry}.

In Section~\ref{sec:integral}, we highlight a connection between $S_J$
(in the two ensembles) and the observable
\begin{equation}
    \begin{split}
        s[\vsig] = \int_0^1 dt \log\left( |\vsig(t)| \right)  \, 
    \end{split}
\end{equation}
over Brownian bridges (or excursions) $\vsig$.
Simple scaling arguments imply that
\begin{equation}
    \begin{split}
        s = \lim_{N \rightarrow \infty} \frac{S_N- \frac{1}{2} N\log N}{N}  \, 
    \end{split}
\end{equation}
is a random variable with a non-trivial limit distribution for large
$N$.
We provide integral formulas for the integer moments of $s$, and we
use these formulas to compute analytically its first two moments in the
two ensembles.

In Section~\ref{sec:combinatorial}, we complement this analysis with a
combinatorial framework at finite size $N$.  We use this second
approach to  provide an effective strategy for the
  computation of finite-size corrections, that we illustrate by
  calculating the first and second moment, for both bridges and
  excursions.  

In particular, we can establish that
\begin{equation}
    \begin{split}
        S_N 
\stackrel{d}{=} \smfrac{1}{2} N \log N + N s + \bigO(\log(N))  \, 
    \end{split}
\end{equation}
where $s$ is a random variable whose distribution depends on the
ensemble (among bridges and excursions).
For Dyck bridges we have
\begin{equation}
    \begin{split}
        \langle s \rangle_{\rm B}&=- \frac{\gamma_E+2}{2} + \bigO\left(\frac{\log N}{\sqrt{N}} \right) \\ 
        \langle s^2 \rangle_{\rm B}&=\frac{4}{3} + \frac{\gamma_E^2}{4} + \gamma_E - \frac{\pi^2}{72} + \bigO\left( \frac{(\log N)^2}{\sqrt{N}} \right)
    \end{split}
\end{equation}
and for Dyck excursions we have
\begin{equation}
    \begin{split}
        \langle s \rangle_{\rm E}&=- \frac{\gamma_E}{2} + \bigO\left(\frac{\log N}{\sqrt{N}} \right) \\ 
        \langle s^2 \rangle_{\rm E}&=\frac{\gamma_E ^2}{4}+\frac{5 \pi ^2}{24} -2 + \bigO\left( \frac{(\log N)^2}{\sqrt{N}} \right)
    \end{split}
\end{equation}
In Section~\ref{sec:numerics} we provide numerical evidence that the
distribution of the rescaled entropy $s$ is non-Gaussian (but we leave
unsolved the question whether the centered distribution for the
excursions is an even function), and we confirm that the predicted values for the first two moments match with the simulated data.
  
With both our approaches it seems possible to access also
higher-order moments, however we cannot prove at present that, for
any finite moment, the evaluation can be performed in closed form, and
that (as we conjecture)
the result is in the form of a rational polynomial that
only involves $\gamma_E$ and (multiple) zeta functions.  
We will investigate these aspects in
future works.

\section{Optimal matchings at $p=1$}

\noindent
In the following, we focus on the EAP with $p=d=1$. We have $N$ white
points $W=\{w_{i}\}_{i=1}^{N}$ and $N$ black points
$B=\{b_{i}\}_{i=1}^{N}$ on a segment $[0,L]$ (i.e., we have `closed'
boundary conditions, instead of `periodic', as we would have if the
points were located on a circle).  We assume that the instance is
generic (i.e., no two coordinates are the same), and we label the
points so that the lists above are ordered ($w_i<w_{i+1}$ and
$b_i<b_{i+1}$).

\subsection{Characterization of optimal matchings} \label{sec:optimalMatch}

We start by giving an alternate, integral representation of the cost
function $H_J(\pi)$.  For a given matching $\pi$, and every
$z \in \mathbb{R}$ which is not a point of $W$ or $B$, define the function
\begin{align}
    k_{\pi}(z) 
&= \sum_{i=1}^N \kappa(z,w_i,b_{\pi(i)})
&
\kappa(z,x,y)=\left\{
    \begin{array}{ll}
        1 & x<z<y \textrm{~or~} y<z<x \\
        0 & \textrm{otherwise}
    \end{array}
\right.
\end{align}
In words, $\kappa(z,x,y)$ is just the indicator function over the
segment with extreme points $x$ and $y$, and $k_\pi(z)$ counts the
number of links of $\pi$ that have endpoints on opposite sides of $z$.

Furthermore, define the function
\begin{align}
    k_{\rm LB}(z) 
&= \big| \#( W \cap [0,z]) - \#( B \cap [0,z] ) \big|
\, ,
\end{align}
where $\#(I)$ denotes the cardinality of the set $I$. 
As well as $k_{\pi}(z)$, also $k_{\rm LB}$ is defined for all $z \in
\mathbb{R} \setminus (W \cup B)$, and counts the excess of black or
white points on the left of $z$.

Then, we have two simple observations
\begin{proposition}
    At $p=1$
    \begin{equation}
        H_J(\pi) = \int \dx{z} k_{\pi}(z) 
        \ef.
    \end{equation}
    Moreover, at $p=1$, for all $\pi$ and all $z$,
    \begin{equation}
k_{\pi}(z) \geq k_{\rm LB}(z) \ef.
    \end{equation}
\end{proposition}
\noindent
This has the immediate corollary that, for all $\pi$,
\begin{equation}
    H_J(\pi) \geq 
    H_J^{\rm LB} :=
    \int \dx{z} k_{\rm LB}(z) 
    \ef.
\end{equation}
Now, call $\pi_{\rm id}$ the identity permutations, that is the so-called \emph{ordered matching}. By simple
inspection, we have that $H_J(\pi_{\rm id}) = H_J^{\rm LB}$.  This
implies that $\pi_{\rm id}$ is optimal, and more generally
\begin{corollary}
    \label{cor.kLb}
    $\pi \in \cZ_J$ iff the functions $k_{\rm LB}$ and $k_{\pi}$ coincide.
\end{corollary}
\noindent
See Figure~\ref{fig:simplematchings} for the description of all optimal matchings at $N=2$.

\begin{figure}
    \centering
    \begin{tikzpicture}
        \node at (-1,0) {$\pi_{1}$};
        \node at (-1,-1.5) {$\pi_{2}$};

        \draw[fill, opacity = 0.1] (-0.5,1.25) -- (8,1.25) -- (8,-2) -- (5.5,-2) -- (5.5,-0.5) -- (-0.5,-0.5) -- (-0.5,1.25);

        \node at (3.75,0.75) {Optimal configurations};

        \begin{scope}
            \draw (0.5,0) arc (0:180:0.25);
            \draw (1.5,0) arc (0:180:0.25);


            \draw[fill=white] (0,0) circle (0.1);
            \draw[fill=black] (0.5,0) circle (0.1);
            \draw[fill=black] (1,0) circle (0.1);
            \draw[fill=white] (1.5,0) circle (0.1);
        \end{scope}

        \begin{scope}[shift={(0,-1.5)}]
            \draw (1.0,0) arc (0:180:0.5);
            \draw (1.5,0) arc (0:180:0.5);

            \draw[fill=white] (0,0) circle (0.1);
            \draw[fill=black] (0.5,0) circle (0.1);
            \draw[fill=black] (1,0) circle (0.1);
            \draw[fill=white] (1.5,0) circle (0.1);
        \end{scope}

        \begin{scope}[shift={(3,0)}]
            \draw (0.5,0) arc (0:180:0.25);
            \draw (1.5,0) arc (0:180:0.25);

            \draw[fill=white] (0,0) circle (0.1);
            \draw[fill=black] (0.5,0) circle (0.1);
            \draw[fill=white] (1,0) circle (0.1);
            \draw[fill=black] (1.5,0) circle (0.1);
        \end{scope}

        \begin{scope}[shift={(3,-1.5)}]
            \draw (1.0,0) arc (0:180:0.25);
            \draw (1.5,0) arc (0:180:0.75);

            \draw[fill=white] (0,0) circle (0.1);
            \draw[fill=black] (0.5,0) circle (0.1);
            \draw[fill=white] (1,0) circle (0.1);
            \draw[fill=black] (1.5,0) circle (0.1);
        \end{scope}

        \begin{scope}[shift={(6,0)}]
            \draw (1.0,0) arc (0:180:0.5);
            \draw (1.5,0) arc (0:180:0.5);

            \draw[fill=white] (0,0) circle (0.1);
            \draw[fill=white] (0.5,0) circle (0.1);
            \draw[fill=black] (1,0) circle (0.1);
            \draw[fill=black] (1.5,0) circle (0.1);
        \end{scope}

        \begin{scope}[shift={(6,-1.5)}]
            \draw (1.0,0) arc (0:180:0.25);
            \draw (1.5,0) arc (0:180:0.75);

            \draw[fill=white] (0,0) circle (0.1);
            \draw[fill=white] (0.5,0) circle (0.1);
            \draw[fill=black] (1,0) circle (0.1);
            \draw[fill=black] (1.5,0) circle (0.1);
        \end{scope}

    \end{tikzpicture}
    \caption{\footnotesize 
        Optimal matchings at $p=1$ for $2N=4$ points.
        $\pi_{1}$ is the matching in which the first white point is matched with the first black point; $\pi_{2}$ is the only other possible matching.
        All optimal configurations satisfy $k_{\pi}(x) = k_{\rm LB}(x)$.}
    \label{fig:simplematchings}
\end{figure}
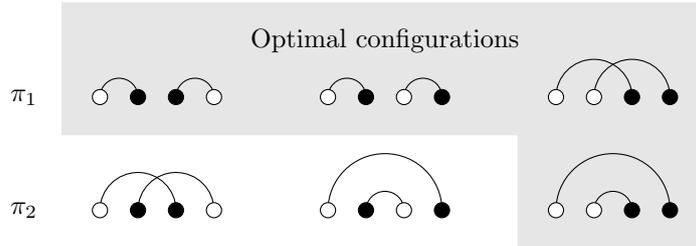

In the following we will provide a simple algorithm to construct the
optimal matchings of a given instance.
In order to do this, we shall now
give another characterization of optimal matchings:
\begin{definition}
\label{def.stack}
    Let $\pi$ be a matching. Let us call $P=(p_1,\ldots,p_{2N})$ the
    ordered list of the points in $W \cup B$.
    For $1\leq i \leq 2N$, we call $P_i(\pi)$ the 
    \emph{stack of $\pi$ at $i$,} that is, the set of points in
    $\{p_1,\ldots,p_i\}$ that are paired by $\pi$ to points in
    $\{p_{i+1},\ldots,p_{2N}\}$ (see Figure~\ref{fig:stackmatching}).
\end{definition}
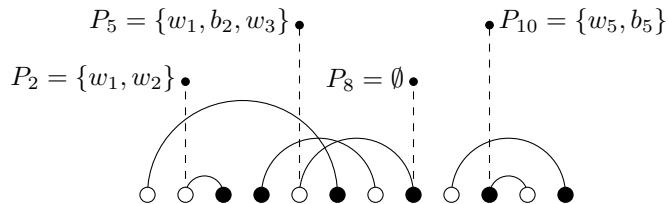
\begin{figure}
    \centering
    \begin{tikzpicture}
        \begin{scope}
            \draw[dashed] (2,0) -- (2,2.25);
            \draw[fill=black] (2,2.25) circle (0.05);
            \node[above left] at (2,2) {$P_5 = \{w_{1}, b_{2}, w_{3}\}$};

            \draw[dashed] (0.5,0) -- (0.5,1.5);
            \draw[fill=black] (0.5,1.5) circle (0.05);
            \node[above left] at (0.5,1.25) {$P_2 = \{w_{1}, w_{2}\}$};

            \draw[dashed] (3.5,0) -- (3.5,1.5);
            \draw[fill=black] (3.5,1.5) circle (0.05);
            \node[above left] at (3.5,1.25) {$P_8 = \emptyset$};

            \draw[dashed] (4.5,0) -- (4.5,2.25);
            \draw[fill=black] (4.5,2.25) circle (0.05);
            \node[above right] at (4.5,2) {$P_{10} = \{w_{5}, b_{5}\}$};

            \draw (1,0) arc (0:180:0.25);
            \draw (2.5,0) arc (0:180:1.25);
            \draw (3,0) arc (0:180:0.75);
            \draw (3.5,0) arc (0:180:0.75);
            \draw (5.5,0) arc (0:180:0.75);
            \draw (5,0) arc (0:180:0.25);

            \draw[fill=white] (0.0,0) circle (0.1);
            \draw[fill=white] (0.5,0) circle (0.1);
            \draw[fill=black] (1.0,0) circle (0.1);
            \draw[fill=black] (1.5,0) circle (0.1);
            \draw[fill=white] (2.0,0) circle (0.1);
            \draw[fill=black] (2.5,0) circle (0.1);
            \draw[fill=white] (3.0,0) circle (0.1);
            \draw[fill=black] (3.5,0) circle (0.1);
            \draw[fill=white] (4.0,0) circle (0.1);
            \draw[fill=black] (4.5,0) circle (0.1);
            \draw[fill=white] (5.0,0) circle (0.1);
            \draw[fill=black] (5.5,0) circle (0.1);
        \end{scope}
    \end{tikzpicture}
    \caption{\footnotesize The stack of a matching at position $i$ is
        given by all the points on the left side of point $i$ (including
        $i$ itself) that are matched to a point on the right side of
        point $i$.  In the picture, $w_{i}$ is the $i$-th white point
        from the right, and analogously $b_{i}$ is the $i$-th black
        point from the right.  At the locations specified by the dashed
    lines, we show the stack of the represented matching.  }
    \label{fig:stackmatching}
\end{figure}
\begin{proposition}
    \label{prop:stack}
    $\pi \in \cZ_J$
    iff, for all $1 \leq i \leq 2N$, the stack of $\pi$ at $i$ is either
    empty or monochromatic, i.e.\ if 
    $P_i(\pi) \cap W = \varnothing$ or $P_i(\pi) \cap B = \varnothing$.
\end{proposition}
\begin{proof}
    Suppose that for, some $1 \leq i \leq 2N$, the stack of $\pi$ at
    $i$ is non-empty and non-monochromatic.  Then there are 
$p_w \in P_i(\pi) \cap W$ and $p_b \in P_i(\pi) \cap B$
    which are matched to 
$q_b \in P_i^c(\pi) \cap B$ and $q_w \in P_i^c \cap W$,
    respectively.  In the matching $\pi'$ in which we swap these two
    pairs, we have $k_{\pi'}(z) = k_{\pi}(z)-2$ for all
    $\max(p_w,p_b)<z<\min(q_w,q_b)$, and $k_{\pi'}(z) = k_{\pi}(z)$
    elsewhere, thus, by Corollary \ref{cor.kLb}, $\pi$ cannot be
    optimal.

    Viceversa, let $\pi$ have only empty or monochromatic stacks.
    This means that, in a right neighbourhood of $p_i$, the
    cardinality of the stack, which by definition coincides with
    $k_\pi(x)$, is exactly given by $k_{\rm LB}$. Furthermore, both these
    functions are constant on the intervals between the points (where
    they jump by $\pm 1$), so the two functions coincide everywhere on
    the domain. This means, by Corollary~\ref{cor.kLb}, that $\pi$ must be optimal.
\end{proof}

\subsection{Enumeration of optimal matchings} \label{sec:enum}

We are now interested in enumerating the optimal matchings at $p=1$
for a fixed configuration $J$ of size $N$.  First of all, we give an
alternative representation of $J$ (already adopted in
\cite{caracciolo2020DyckBoundConcave}) that will be useful in the
following.  A configuration $J$ can be encoded by sorting the $2N$
points in order of increasing coordinate (as in Definition
\ref{def.stack} above), and defining
\begin{itemize}
    \item a vector of spacings 
$\vec{s}(J) \in (\mathbb{R}^+)^{2N}$, given by
      $\vec{s}(J)=(p_1,p_2-p_1,p_3-p_2,\ldots,p_{2N}-p_{2N-1})$;
    \item a vector of signs, $\vsig(J) \in \{-1,+1\}^{2N}$ such
        that if $p_i$ is white (resp.\ black),
        $\sigma_{i}=+1$ (resp.\ $-1$). Note that $\sum_i \sigma_i = 0$.
\end{itemize}
It is easily seen that the criterium in Proposition \ref{prop:stack}
is stated only in terms of $\vsig(J)$. This makes clear that the set
$\cZ_J$ itself is fully determined by $\vsig(J)$.  Thus, from this
point onward, we understand that $\cZ(\vsig)$, $Z(\vsig)$ and
$S(\vsig)$ are synonims of the quantities $\cZ_J$, $Z_J$ and $S_J$,
for any $J$ with $\vsig(J)=\vsig$.

Binary vectors can be represented as lattice paths, i.e.\ paths in the
plane, starting at the origin and composed by up-steps (or rises)
$(+1,+1)$ and down-steps (or falls) $(+1,-1)$.  So we have a bijection
among zero-sum binary vectors $\vsig$ and lattice bridges, in which
the $i$-th step of the path is $(i,\sigma_i)$.  The bijection between
color orderings, binary vectors and lattice paths is so elementary
that in the following, with a slight abuse of notation, we will just
identify the three objects.

We are interested in two classes of binary vectors\;/\;lattice paths:
\begin{itemize}
    \item \textbf{Dyck bridges} $\mathcal{B}_N$ of semi-lenght (size)
      $N$. These are lattice paths with an equal number of up- and
      down-steps. They are precisely in bijection with the color
      orderings of $N$ white and $N$ black points, i.e.\ with the
      color ordering of all the possible configurations $J$.  There
      are $B_N = |\mathcal{B}_N| = \binom{2N}{N}$ Dyck bridges of size
      $N$.
    \item \textbf{Dyck paths} $\mathcal{C}_N$ of semi-length (size)
      $N$. These are lattice bridges that never reach negative
      ordinate (we will also call them \emph{excursions}, in analogy
      with their continuum counterparts).  They are in bijection with
      configurations $J$ in the forementioned ``Dyck excursions''
      ensemble.  There are $C_N = |\mathcal{C}_N| = \frac{1}{N+1} B_N$
      Dyck paths of size $N$.
\end{itemize}
In the following, the generating function of the series $B_N$ and
$C_N$ will turn useful. We have
\begin{equation}
\label{eq.BzCz}
    \begin{split}
        B(z) &= \sum_{N \geq 0} B_N z^N = (1-4z)^{-\frac{1}{2}}  \, ,\\
        C(z) &= \sum_{N \geq 0} C_N z^N = \frac{1-\sqrt{1-4z}}{2z}  \, .
    \end{split}
\end{equation}
Our notion of height will be associated to the steps of the path. We
call $h_{i}(\vsig)$ the \emph{height of the path at step $i$}, that
is, the height of the midpoint of the $i$-th step of the path, that in
terms of the binary vector reads
\begin{equation}
    \begin{split}
        h_{i}(\vsig) = 
        \sigma_1 + \sigma_2 + \cdots + \sigma_{i-1} + 
        \frac{\sigma_{i}}{2} \, .
    \end{split}
\end{equation}
The choice of the midpoint to compute the height is arbitrary, but has
the advantage of being a symmetric definition with respect to
reflections w.r.t.\ the $x$-axis, while taking values in an equispaced range of integers.
We can then define $\bar{h}_i$ as the positive integers 
\begin{equation}
    \label{def.hb}
    \bar{h}_i=|h_i|+\frac{1}{2}
    \ef.
\end{equation}
Then we have
\begin{lemma}
    \label{lem.ZJ}
    \begin{equation}
        Z(\vsig)
        = \prod_{\substack{
                i=1,\ldots,2N \\
                h_i \sigma_i < 0
        }} \bar{h}_i(\vsig)
        \ef.
        \label{eq:ZJfinal}
    \end{equation}
\end{lemma}
\begin{proof}
The proof goes through the characterisation of the stacks of optimal
configurations, given in Proposition \ref{prop:stack}. First of all,
notice that the list of positions given by the condition $h_i \sigma_i
< 0$ is the one at which the stack at $i-1$ decreases its size,
because one point in the stack is paired to the $i$-th point. We shall
call \emph{closing steps} the elements of this set (and \emph{closing
  points} the associated points in $P$), and \emph{opening steps}
those in the complementary set.  Indeed, the cardinality of the stack
at $i-1$ is exactly $\bar{h}_i$, while the sign of $h_i$ determines
the colour of the points in the stack at $i-1$ (which is also the
colour of the stack at $i$, unless the latter is empty). So, there are
exactly $\bar{h}_i$ choices for the pairing at $i$, while, if $i$ is
not in the list above, the choice is unique. As the cardinalities of
the stacks are the same for all optimal configurations, the choice at
$i$ does not affect the number of possible choices at $j>i$, and we
end up with Equation~\eqref{eq:ZJfinal}.
\end{proof}     

Notice that Equation~\eqref{eq:ZJfinal} is trivially equivalent to
\begin{equation}
    \begin{split}
        Z(\vsig)
        = \prod_{i=1}^{2N} \sqrt{\bar{h}_i(\vsig)}
        = \prod_{\substack{
                i=1,\ldots,2N \\
                \sigma_i = -1
        }} \bar{h}_i(\vsig)
        \, .
    \end{split}
\end{equation}
The proof above has a stronger implication: we can construct the
$m$-th of the $Z_J$ solutions by a polynomial-time algorithm (which
takes on average time $\sim N \log N$ and space $\sim \sqrt{N}$ if the
suitable data structure is used), despite the fact that, as is
apparent from Lemma \ref{lem.ZJ}, the typical values of $Z_J$ are
potentially at least exponential in $N$. The algorithm goes as
follows. First, rewrite $m$ in the form
$m-1=a_1 + a_2 \bar{h}_1 + a_3 \bar{h}_1 \bar{h}_2 + 
\cdots + a_N \bar{h}_1 \bar{h}_2 \cdots \bar{h}_{N-1}$,
with $0 \leq a_j < \bar{h}_j$. Then, say that $i(j)=i$ if the $j$-th
closing point is $p_i$. Now, produce the $N$ pairs of the $m$-th
optimal matching by pairing the closing points, in order of increasing
$j$, by pairing this point to the $a_j$-th of the stack in $i(j)-1$,
when this is sorted (say) in increasing order.

\section{Statistical properties of $S(\sigma)$} \label{sec:StatSN}

We are now interested in the statistical properties of the entropy
\begin{equation}\label{eq:s}
    \begin{split}
        S_N(\vsig) = \log Z(\vsig) = \frac{1}{2} \sum_{i=1}^{2N} \log(\bar{h}_i(\vsig))
         \,
    \end{split}
\end{equation}
when $\vsig(J)$ is a random variable induced by some probability
measure on the space of configurations $J$ of $2N$ points, and $N$
tends to infinity.  The subscript $N$ reminds us that we are at size
$|W|=|B|=N$.

In matching problems, the typical choices for the configurational
probability measure are factorized over a measure on the spacings and
a measure on the color orderings, i.e.
\begin{equation}
    \begin{split}
\mu(J) = \mu_{\rm spacing}(\vec{s}(J)) \; \mu_{\rm color}(\vsig(J)) \, 
    \end{split}
\end{equation}
(see \cite{caracciolo2020DyckBoundConcave} for more details and
examples).  As $S_N(J) = S_N(\vsig(J))$, we can again forget about the
spacing degrees of freedom, and study the statistics of $S_N(\vsig)$
induced by some measure $\mu_{\rm color}(\vsig)$.  In particular, we
will study the cases in which $\vsig$ is uniformly drawn from the set
of Dyck paths, or uniformly drawn from the set of Dyck bridges.

\subsection{Integral formulas for the integer moments of $S(\sigma)$ via Wiener processes}
\label{sec:integral}

It is well known (see Donsker's theorem \cite{donsker1951invariance})
that lattice paths such as Dyck paths and bridges converge, as
$N\rightarrow\infty$ and after a proper rescaling, to Brownian bridges
and Brownian excursions.  Brownian bridges are Wiener processes
constrained to end at null height, while Brownian excursions are
Wiener processes constrained to end at null height and to lie in the
upper half-plane.  The correct rescaling of the steps of the lattice
paths that highlights this convergence is given by $(+1,\pm1)
\rightarrow \left( +\frac{1}{N}, \pm\frac{1}{\sqrt{N}} \right) $.

These scalings suggest to consider a rescaled version of the entropy 
\begin{equation}
    \begin{split}
        s(\vsig) 
        = \frac{S_N(\vsig) - \frac{1}{2}N\log N}{N} 
        = \frac{1}{2N} \sum_{i=1}^{2N} \log\left(\frac{\bar{h}_i(\vsig)}{\sqrt{N}}\right)
        \, .
    \end{split}
\end{equation}
In the limit $N\rightarrow\infty$, the rescaled entropy will converge to an integral operator over Wiener processes
\begin{equation}
    \begin{split}
        s[\vsig] = \int_0^1 dt \, \log\left( |\vsig(t)| \right) 
    \end{split}
\end{equation}
where $\vsig(x)$ is a Brownian bridge/excursion.

The integer moments of $s[\vsig]$ can be readily computed as correlation functions of the Brownian process:
\begin{equation} \label{eq:WienerGeneral}
    \begin{split}
        \langle (s[\vsig])^k \rangle_{\rm B/E}
        &= \int 
\mathcal{D}_{\rm B/E}[\vsig] \int_0^1 dt_1 \dots dt_k \prod_{a=1}^k \log\left( |\vsig(t_a)| \right) \\
        &= k! \int_{\Delta_k} dt_1 \dots dt_k \int_{\mathbb{R}} dx_1 \dots dx_k \prod_{a=1}^k \log\left( |x_a| \right)
        \int 
\mathcal{D}_{\rm B/E}[\vsig] \prod_{a=1}^k \delta( \vsig(t_a) - x_a )
         \, ,
    \end{split}
\end{equation}
where $\Delta_k \subset \mathbb{R}^k$ is the canonical symplex 
$\{0 = t_0 < t_1 < t_2 < \cdots < t_k < t_{k+1} = 1\}$, and
$\mathcal{D}_{\rm B/E}[\vsig]$ is the standard measure on the Brownian
process of choice among bridges and excursions.  The last integral is
the probability that the Brownian process we are interested in starts
an ends at the origin and visits the points
$(t_1,x_1),\dots,(t_k,h_k)$, while subject to its constraints.

Let us consider Brownian bridges first. In this case, the probability
that a Wiener process travels from $(t_i,x_i)$ to $(t_f,x_f)$ is given
by $\cN( x_f-x_i | 2(t_f-t_i))$ where $\cN(x | \sigma^2 )$ is the
p.d.f.\ of a centered Gaussian distribution with variance $\sigma^2$.
The factor $2$ comes from Donsker's theorem, and is due to the fact
that the variance of the distribution of the steps in the lattice
paths is exactly 2.  Thus, for Brownian bridges
\begin{equation}
    \begin{split}
        \int \mathcal{D}_{\rm B}[\vsig] \prod_{a=1}^k \delta( \vsig(t_a) - x_a ) 
        = \frac{\sqrt{4 \pi}}{\prod_{a=0}^{k} \sqrt{4 \pi (t_{a+1} - t_{a} ) }} \exp \left[ - \sum_{a=0}^k \frac{(x_{a+1} - x_{a})^2}{4(t_{a+1} - t_{a})}  \right] 
         \, 
    \end{split}
\end{equation}
where $x_0=0$, $x_{k+1}=0$ and the factor $\sqrt{4 \pi}$ is a normalization, 
so that
\begin{equation}\label{eq:integralB}
    \begin{split}
        \langle (s[\vsig])^k \rangle_{\rm B}
        &= k! \int_{\Delta_k} dt_1 \dots dt_k \int_{\mathbb{R}} dx_1 \dots dx_k \frac{\sqrt{4\pi}\prod_{a=1}^k \log\left( |x_a| \right)}{\prod_{a=0}^{k} \sqrt{4 \pi (t_{a+1} - t_{a} ) }} \exp \left[ - \sum_{a=0}^k \frac{(x_{a+1} - x_{a})^2}{4(t_{a+1} - t_{a})}  \right] 
\, .
    \end{split}
\end{equation}
Brownian excursions can be treated analogously using the reflection principle.
In this case, the conditional probability that a Wiener process travels from
$(t_i,x_i)$ to $(t_f,x_f)$ without ever reaching negative heights, given that it already reached $(t_i ,x_i )$, is
given by 
$\cN( x_f-x_i | 2 \Delta (t_f - t_i) ) - \cN( x_f+x_i | 2 ( t_f -t_i ) )$
for $x_{i,f} > 0$, while for $x_{i}=0$ and $x_f = x$ (or viceversa) it equals 
$\frac{|x|}{2 (t_f -t_i  )} \cN(x| 2 (t_f -t_i ))$.
Moreover, now all $x_i$'s are constrained to be positive.
Thus, for Brownian excursions
\begin{equation}\label{eq:integralE}
    \begin{split}
        \langle (s[\vsig])^k \rangle_{\rm E}
        &= k! \int_{\Delta_k} dt_1 \dots dt_k \int_{[0,+\infty)} dx_1 \dots dx_k \frac{\sqrt{4\pi} x_1 x_k \prod_{a=1}^k \log\left( x_a \right)}{t_1 (1-t_k) \prod_{a=0}^{k} \sqrt{4 \pi (t_{a+1} - t_{a} ) }} 
\\
        &\quad\times
        \exp\left[ -\frac{x_1^2}{4 t_1} -\frac{x_k^2}{4 (1- t_k)} \right] 
        \prod_{a=1}^{k-1} \left\{ \exp \left[ -  \frac{(x_{a+1} - x_{a})^2}{4(t_{a+1} - t_{a})}  \right]  - \exp \left[ - \frac{(x_{a+1} + x_{a})^2}{4(t_{a+1} - t_{a})}  \right]  \right\}
\, .
    \end{split}
\end{equation}

In both cases, the Gaussian integrations on the heights $x_i$ can be explicitly performed. 
    First of all, we replace
    \begin{equation}
        \begin{split}
            \log|x_a |  = \frac{1}{2} \partial_{\kappa_a} \left[ x_a ^{2\kappa_a} \right]_{\kappa_a =0} 
            \, .
        \end{split}
    \end{equation}
    Then, we treat the contact terms.
    In the case of bridges, the contact terms can be rewritten as
    \begin{equation}
        \begin{split}
            \exp \left[ \frac{x_{a+1} x_a }{2( t_{a+1} -t_a  )} \right] 
            = \cosh\left( \frac{x_{a+1} x_a }{2( t_{a+1} -t_a  )}  \right)
            \, ,
        \end{split}
    \end{equation}
    where the hyperbolic sine term is discarded due to the parity of the rest of the integrand in the variables $x_a$.
    In the case of excursions, the contact term instead reads
    \begin{equation}
        \begin{split}
            \exp \left[ \frac{x_{a+1} x_a }{2( t_{a+1} -t_a  )} \right] -
            \exp \left[ - \frac{x_{a+1} x_a }{2( t_{a+1} -t_a  )} \right] 
            = 2\sinh\left( \frac{x_{a+1} x_a }{2( t_{a+1} -t_a  )}  \right)
            \, .
        \end{split}
    \end{equation}
    In both cases, we can expand the hyperbolic function in power-series, so that the integrations in the $x_a$ variables are now factorized and of the kind
    \begin{equation}
        \begin{split}
            \int_{\mathbb{R}} dx \, x^{2k} \exp\left[ -\frac{x^2}{\lambda} \right] = \Gamma\left( k+\frac{1}{2} \right) \lambda^{k+\frac{1}{2}} 
            \, 
        \end{split}
    \end{equation}
    (in the case of excursions, a factor $1/2$ must be added to take into account the halved integration domain).
    
    Using these manipulations, the first two moments for both bridges and excursions can be analytically computed. We detail the computations in Appendix~\ref{app:momint}. The results are:
    \begin{equation}\label{eq:M1wiener}
    \begin{split}
        \langle s[\vsig] \rangle_{\rm B} &= -\frac{\gamma_E+2}{2} \\ 
        \langle s[\vsig] \rangle_{\rm E} &= -\frac{\gamma_E}{2} \, ,
    \end{split}
\end{equation}
where $\gamma_E$ is the Euler--Mascheroni constant,
and
\begin{equation}
    \begin{split}
        \left\langle (s[\vsig])^2 \right\rangle_{\rm B}  &= \frac{4}{3}+\gamma_E +\frac{\gamma_E ^2}{4}-\frac{\pi ^2}{72}\\
        \left\langle (s[\vsig])^2 \right\rangle_{\rm E}  &= \frac{\gamma_E ^2}{4}+\frac{5 \pi ^2}{24}-2
         \, .
    \end{split}
\end{equation}

The approach presented in this Section is simple in spirit, and allows
to connect our problem to the vast literature on Wiener processes.
Moreover, it is suitable for performing Monte Carlo numerical
integration to retrieve the moments of $s(\vsig)$.  

In this Section we worked directly in the continuum limit. 
In the next Section, we provide a combinatorial approach that allows to recover the values of the first two moments in a discrete setting, and to compute finite-size corrections in the limit $N\to\infty$.


\subsection{Combinatorial properties of the integer moments of $S(\sigma)$ at finite $N$}
\label{sec:combinatorial}

In this section, we introduce a combinatorial method to compute the
moments of $S_N(\vsig)$ in the limit $N\rightarrow\infty$.  This new
approach allows to retain informations on the finite-size corrections.

The underlying idea is to reproduce Equation~\eqref{eq:WienerGeneral}
in the discrete setting for the variable $S_N(\vsig)$, and to study
its large-$N$ behaviour using methods from analytic combinatorics.

We start again from
\begin{equation}
    \begin{split}
        S_N(\vsig) = \sum_{\substack{i=1\dots 2N\\\sigma_i=-1}} \log(\bar{h}_i(\vsig))
         \,
    \end{split}
\end{equation}
In the following, the superscript/subscript $\rm T=E,B$ will stand for
Dyck paths (excursions, E) and Dyck bridges (B) respectively, $T_N =
C_N,B_N$ and $\mathcal{T}_N = \mathcal{C}_N, \mathcal{B}_N$; we will
mantain the notation unified whenever possible.

The $k$-th integer moment equals
\begin{equation}\label{eq:Mk}
    \begin{split}
        \langle S_N(\vsig) \rangle_{\rm T} 
        &:= 
        M_{N,k}^{\rm (T)} 
        = \frac{1}{T_N} \sum_{\vsig \in \mathcal{T}_N} \left[ S_N(\vsig) \right]^k
        = \frac{k!}{T_N} 
        \sum_{\vsig \in \mathcal{T}_N} 
        \sum_{\substack{
1 \leq t_1 , \dots, t_k \leq 2N \\ 
\sigma_{t_1} = \dots = \sigma_{t_k} = -1} } 
\prod_{a=1}^k \log(\barh_{t_a}(\vsig))
        \\
        &= k!
        \sum_{c=1}^k
        \sum_{
            \substack{
                1 \leq t_1 < t_2 < \cdots < t_c \leq 2N
                \\
                \nu_1,\ldots,\nu_c \geq 1
\\
                \nu_1 + \cdots = \nu_c = k
                \\
                \bar{h}_1, \ldots, \bar{h}_c >0
        }}
        \prod_{a=1}^c
        \left( \frac{
                \left( \log \bar{h}_a \right)^{\nu_a}
            }{\nu_a !}
        \right)
        \frac{\cM_N^{\rm (T)}(t_1,\cdots,t_c;\bar{h}_1,\cdots,\bar{h}_c)}{T_N}
        \, 
    \end{split}
\end{equation}
where 
$\mathcal{M}^{\rm (T)}_N(t_1, \dots, t_c;\barh_1, \dots, \barh_c)$ is
the number of paths of type $\rm T$ that has closing steps at
horizontal positions $t_1, \dots t_c$, and at heights $h_1 = \pm
(\barh_1 - 1/2), \dots, h_c = \pm (\barh_c - 1/2)$.

The last equation reproduces, as anticipated earlier,
Equation~\eqref{eq:WienerGeneral} in the discrete setting.  Notice
that here we must take into account the multiplicities $\nu_a$, while
in the continuous setting we could just set $c=k$ and $\nu_a=1$ for
all $a$, as the contribution from the other terms is washed out in the
continuum limit.
This suggests that, in this more precise approach, we will verify
explicitly that the leading contributions in the large-$N$ limit
comes from the $c=k$ term of Equation~\eqref{eq:Mk}.

In order to study Equation~\eqref{eq:Mk}, we take the following route.
As this equation depends on $N$ only implicitly through the
summation range, and explicitly through a normalization, we would like
to introduce a generating function
\begin{equation}
    \label{eq:Mkz}
    M_{k}^{\rm (T)}(z) = \sum_{N \geq 1} z^N T_N M_{N,k}^{\rm (T)} 
\ef 
\end{equation}
that will decouple the summation range over the variables
$t_{i+1}-t_i$.  By singularity
analysis~\cite{flajolet2009AnalyticCombinatorics}, the asymptotic
expansion for $N \rightarrow \infty$ of $M^{\rm (T)}_{N,k}$ will be
then retrieved by the singular expansion of $M_k^{\rm (T)}(z)$ around
its dominant singularity.

We start by giving an explicit form for $\mathcal{M}^{\rm (T)}_N$ for Dyck paths and Dyck bridges.
\begin{proposition}\label{prop:MNth}
In the case of Dyck bridges, we have
\begin{equation}\label{eq:MBcursive}
    \begin{split}
        &\mathcal{M}^{\rm (B)}_N (t_1,\cdots,t_c;\bar{h}_1,\cdots,\bar{h}_c)
        \\
        &\quad= 
        2 B_{t_1-1,\barh_1} 
        \left( B_{t_2 - t_1 -1, \barh_2 - (\barh_1 -1)} + B_{t_2 - t_1 -1, \barh_2 + (\barh_1 -1)}  \right)\cdots
        \\
        &\qquad
        \cdots
        \left( B_{t_c - t_{c-1} -1, \barh_c - (\barh_{c-1} -1)} + B_{t_c - t_{c-1} -1, \barh_c + (\barh_{c-1} -1)}  \right) 
        B_{2N - t_c, \barh_c -1}
        \, ,
    \end{split}
\end{equation}
where
\begin{equation}
    \begin{split}
        B_{a,b} = 
        \begin{cases}
            \binom{a}{\frac{a+b}{2}} & \text{if } a,b \in \mathbb{Z}^+ \text{ and } a+b \text{ is even} \\
            0 & \text{otherwise}
        \end{cases}
    \end{split}
\end{equation}
is the number of unconstrained paths that start at $(x,y)$ and end at $(x+a,y+b)$.

In the case of Dyck paths, we have
\begin{equation}\label{eq:MEcursive}
    \begin{split}
        &\mathcal{M}^{\rm (E)}_N (t_1,\cdots,t_c;\bar{h}_1,\cdots,\bar{h}_c)
        \\
        &\quad
        = 
        C_{t_1-1,\barh_1,0} \, C_{t_2-t_1-1,\barh_2 - (\barh_1-1), \barh_1-1}
        \cdots\\
        &\qquad\cdots
        C_{t_c-t_{c-1}-1, \barh_c - (\barh_{c-1} -1), \barh_{c-1}-1}
        C_{2N-t_c, -(\barh_c-1), \barh_c-1} 
    \end{split}
\end{equation}
where
\begin{equation}
    \begin{split}
        C_{a,b,d} = 
        \left( B_{a,b} - B_{a,b+2(d+1)}  \right)\theta(b+d) \, \qquad a,b,d \in \mathbb{Z}^+ \, ,
    \end{split}
\end{equation}
is the number of paths that start at $(x,y)$, end at $(x+a,y+b)$ and never fall below height $y-d$, and 
$\theta(x)=1$ for $x\geq0$ and zero otherwise, 
\end{proposition}
\noindent
Notice that, while in general the $\theta$ factors are necessary for
the definition of $C_{a,b,d}$ in terms of $B_{a,b}$, in our specific
case they are all automatically satisfied, as $\barh_a \geq 1$ for all
$1\leq a \leq c$.
\begin{proof}
    Let us start by considering Dyck bridges.
    The idea is to decompose a path contributing to the count of 
    $\mathcal{M}^{\rm (B)}_N (t_1,\cdots,t_c;\bar{h}_1,\cdots,\bar{h}_c)$
    around its closing steps:
    \begin{itemize}
        \item the first closing step starts at coordinate
          $\left(t_1-1,\pm \barh_1 \right)$.  There are
          $B_{t_1-1,\barh_1} + B_{t_1-1,-\barh_1} =
          2B_{t_1-1,\barh_1}$ different portions of path joining the
          origin to the starting point of the first closing step.
        \item the $a$-th closing step happens $t_a - t_{a-1}-1$ steps
          after the $(t-1)$-th one, and, based on the relative sign of
          the heights of the two closing steps, their difference in
          height equals $\barh_a - (\barh_{a-1}-1)$ or $\barh_a +
          (\barh_{a-1}-1)$. Thus, there are $B_{t_a - t_{a-1} -1,
            \barh_a - (\barh_{a-1}-1)} + B_{t_a - t_{a-1} -1, \barh_a
            + (\barh_{a-1}-1)}$ different portions of path connecting
          the two closing steps.
        \item the last closing step happens $2N-t_c$ steps before the
          end of the path and at height $h_c = \pm \left( \barh_c -
          \frac{1}{2} \right)$. Thus, there are $B_{2N-t_c,\barh_c-1}$
          portions of path concluding the original path.
    \end{itemize}
    The product of the contribution of each subpath recovers Equation~\eqref{eq:MBcursive}.
    
    The case of Dyck paths can be treated analogously, with a few
    crucial differences.  In fact, each of the portions of path
    between the $i$-th and $(i+1)$-th closing steps (which, for
    excursions, are just down-steps) has now the constraint that it
    must never fall below the horizontal axis, i.e.\ must never reach
    a height $\left( \barh_i -\frac{1}{2} \right)$ lower with respect
    to its starting step.  Let us count these paths.  A useful trick
    to this end is the discrete version of the \emph{reflection method}, that we already used in Section~\ref{sec:integral}.  Call $a$
    the total number of steps, $b$ the relative height of the final
    step with respect to the starting step, and $c$ the maximum fall
    allowed with respect to the starting step.  Moreover, call
    \emph{bad paths} all paths that do not respect the last
    constraint.  A bad path is characterized by reaching relative
    height $-c-1$ at some point (say, the first time after $s$ steps).
    By reflecting the portion of path composed of the first $s$ steps,
    we obtain a bijection between bad paths and unconstrained paths
    that start at relative height $-2(c+1)$, and reach relative height
    $b$ after $a$ steps.  Thus, the total number of good paths
    $C_{a,b,d}$ is given by subtraction as
    \begin{equation}
        \begin{split}
            C_{a,b,d} = B_{a,b} - B_{a,b+2(d+1)}
            \, .
        \end{split}
    \end{equation}
    This line of thought holds for all values of $a,d > 0$ and $b \geq
    -d$; if $b < -d$ we just have $C_{a,b,d} = 0$.  Moreover, by
    properties of $B_{a,b}$, $C_{a,b,d}=0$ if $a+b$ is not an even
    number.

    Equation~\eqref{eq:MEcursive} can be easily established by
    decomposing a generic (marked) path around its closing steps, and
    by applying our result above.
\end{proof}
The fact that we want to exploit now is that, while a given binomial
factor $B_{a,b}$ (and its constrained variant $C_{a,b,d}$) are not easy
to handle exactly, their generating function in $a$ have simple
expressions, induced by analogously simple decompositions, that we
collect in the following:
\begin{proposition}\label{prop:fgenBC}
    \begin{align} \label{eq:Bbz}
        B_b(z) 
        &:= \sum_a z^{\frac{a}{2}} B_{a,b}
        = B(z) (\sqrt{z} C(z))^{|b|}    
        \ef ,
        \\
        \label{eq:Cbdz}
        C_{b,d}(z)
        &:= \sum_a z^{\frac{a}{2}} C_{a,b,d} 
        = B(z) \left[ (\sqrt{z} C(z))^{|b|} - (\sqrt{z} C(z))^{|b+2(d+1)|}   \right] \theta(b+d)
        \ef, 
        \\
        \label{eq:Cbz}
        C_b(z) 
          &:= \sum_a z^{\frac{a}{2}} C_{a,b,0}
          = B(z) \left( 1-zC(z)^2 \right)  (\sqrt{z} C(z))^b \theta(b)
          \ef ,
    \end{align}
    where, as in (\ref{eq.BzCz}),
    \begin{align}
        B(z)
        &=\sum_{k \geq 0} z^{k} B_k 
        = 
        \frac{1}{\sqrt{1-4z}}
        \ef;
        &
        C(z)
        =\sum_{k \geq 0} z^{k} C_k
        = 
        \frac{1-\sqrt{1-4z}}{2z}
        \ef.
    \end{align}
\end{proposition}
\begin{proof}    
    To obtain Equation~\eqref{eq:Bbz}, observe that a path going from
    $(0,0)$ to $(a,b)$ with non-negative $b$ can be uniquely
    decomposed as $w = w_{0} u_{1} w_{1} u_{2} w_{2} \dots u_{h}
    w_{h}$ where $u_{i}$ is the right-most up-step of $w$ at height
    $i-1/2$, and $w_{i}$ is a (possibly empty) Dyck path, for all
    $i=1,\dots,h$, while $w_{0}$ is a (possibly empty) Dyck bridge.
    Thus,
    \begin{equation}
        \begin{split}
            B_{a,b} = \sum_{\substack{\ell_0 , \dots, \ell_b \geq 0 \\ 2 \sum_{i=0}^b \ell_i + b = a }} B_{\ell_0} C_{\ell_1} \cdots C_{\ell_b}
            \, .
        \end{split}
    \end{equation}
    For negative $b$, the same reasoning holds with $u_i$'s replaced
    by down-steps, hence the absolute value on $b$ in the result.
    Equation~\eqref{eq:Bbz} then follows easily.

    Equation~\eqref{eq:Cbdz} follows from $C_{a,b,d} = \left( B_{a,b} - B_{a,b+2(d+1)} \right) \theta(b+d)$.

    Equation~\eqref{eq:Cbz} can be derived either as a special case of
    Equation~\eqref{eq:Cbdz}, or as a variation of $\eqref{eq:Bbz}$
    where $w_0$ must be a Dyck path.  The equivalence of these two
    decompositions is granted by the fact that $C(z) = \left( 1-z
    C(z)^2 \right) B(z)$.
\end{proof}

Let us introduce the symbol $x=x(z)$ for the recurrent quantity
\begin{equation}
    x(z) = z C(z)^2 = C(z)-1
\end{equation}
which, if used to parametrise the other relevant quantities, gives
\begin{align}
    z(x) &= \frac{x}{(1+x)^2}
    \ef;
      &
    B(z(x)) &= \frac{1+x}{1-x}
    \ef.
\end{align}
Then, Equation~\eqref{eq:Mkz} reads
\begin{equation}\label{eq:Mkz2}
    \begin{split}
        M_{k}^{\rm (T)}(z)
&= k!
\sum_{c=1}^k
\sum_{
    \substack{
        \nu_1,\ldots,\nu_c \geq 1\\
        \bar{h}_1, \ldots, \bar{h}_c >0\\
        \sum_a \nu_a=k
}}
\prod_{a=1}^c
\left( \frac{
        \left( \log \bar{h}_a \right)^{\nu_a}
    }{\nu_a !}
\right)
\cM^{\rm (T)}(z;\bar{h}_1,\cdots,\bar{h}_c)
\ef,
\end{split}
\end{equation}
where
\begin{equation}
    \begin{split}
        \cM^{\rm (T)}(z;\barh_1, \cdots, \barh_c) 
        = 
        \sum_{N\geq0}z^N 
        \sum_{1 \leq t_1 < t_2 < \cdots < t_c \leq 2N} \cM^{\rm (T)}_N(t_1,\cdots,t_c; \barh_1, \cdots, \barh_c)
        \, .
    \end{split}
\end{equation}
\begin{proposition}
    Using $x$ to denote $x(z)$, we have that for bridges
    \begin{multline}
        \label{eq:Mkz2a}
        \quad
        \cM^{\rm (B)}(z;\bar{h}_1,\cdots,\bar{h}_c)
        =
        2 z^{\frac{c}{2}}
        B(z)^{c+1} \sqrt{x}^{\,\bar{h}_1}
        \big(
            \sqrt{x}^{\,|\bar{h}_2-\bar{h}_1+1|}
            +
            \sqrt{x}^{\,\bar{h}_2+\bar{h}_1-1}
        \big)
        \\
        \cdots
        \big(
            \sqrt{x}^{\,|\bar{h}_c-\bar{h}_{c-1}+1|}
            +
            \sqrt{x}^{\,\bar{h}_c+\bar{h}_{c-1}-1}
        \big)
        \sqrt{x}^{\,\bar{h}_c-1}
        \ef,
        \quad
    \end{multline}
    and for excursions
    \begin{multline}
        \label{eq:Mkz2b}
        \quad
        \cM^{\rm (E)}(z;\bar{h}_1,\cdots,\bar{h}_c)
        =
        z^{\frac{c}{2}}
        B(z)^{c+1} (1-x) \sqrt{x}^{\,\bar{h}_1}
        \big(
            \sqrt{x}^{\,|\bar{h}_2-\bar{h}_1+1|}
            -
            \sqrt{x}^{\,\bar{h}_2+\bar{h}_1+1}
        \big)
        \\
        \cdots
        \big(
            \sqrt{x}^{\,|\bar{h}_c-\bar{h}_{c-1}+1|}
            -
            \sqrt{x}^{\,\bar{h}_c+\bar{h}_{c-1}+1}
        \big)
        (1-x)
        \sqrt{x}^{\,\bar{h}_c-1}
        \ef.
        \quad
    \end{multline}
\end{proposition}
\begin{proof}
    First of all, we notice that
    \begin{equation}
        \begin{split}
            &\cM^{\rm (T)}_N(t_1, \cdots, t_c; \barh_1, \cdots, \barh_c) \\
            &\quad= 
            f_1(t_1-1;\barh_1)f_2(t_2-t_1-1;\barh_2,\barh_1)\cdots f_c(t_c-t_{c-1}-1;\barh_c,\barh_{c-1}) f_{c+1}(2N-t_c;\barh_c)
            \, 
        \end{split}
    \end{equation}
    for some functions $f_i$ that depend on the type of paths $\rm T$ that we are studying.
    Thus, by performing the change of summation variables $\{t_1,\cdots,t_c,N\} \rightarrow \{\alpha_1, \cdots, \alpha_{c+1}\}$ such that
    \begin{equation}
        \begin{split}
            &\alpha_1 = t_1-1 \, , \\
            &\alpha_i = t_i - t_{i-1} -1 \ef , \qquad 2 \leq i \leq c \, ,\\
            &\alpha_{c+1} = 2N-t_c \, ,
            \, 
        \end{split}
    \end{equation}
    we have that
    \begin{equation}
        \begin{split}
            &\cM^{\rm (T)}(z;\barh_1, \cdots, \barh_c) 
            \\
            &\quad= 
            \sum_{N\geq0}z^N 
            \sum_{1 \leq t_1 < t_2 < \cdots < t_c \leq 2N} \cM^{\rm (T)}_N(t_1,\cdots,t_c; \barh_1, \cdots, \barh_c)
            \\
            &\quad=
            z^{c/2}\sum_{\alpha_1, \cdots, \alpha_{c+1} \geq 0}
            f_1(\alpha_1;\barh_1)z^{\alpha_1/2}
            \cdots
            f_{c}(\alpha_{c};\barh_c,\barh_{c-1})z^{\alpha_2/2}
            f_{c+1}(\alpha_{c+1};\barh_c)z^{\alpha_{c+1}/2}
            \, ,
        \end{split}
    \end{equation}
    so that all summations are now untangled.
    Equations~\eqref{eq:Mkz2a} and~\eqref{eq:Mkz2b} can now be recovered by using the explicit form of the functions $f_i$ for Dyck paths and Dyck bridges given in Proposition~\ref{prop:MNth}, and the analytical form for the generating functions given in Proposition~\ref{prop:fgenBC}.
    Again, notice that the $\theta$ functions are all automatically satisfied as $\barh_i \geq 1$ for all $1 \leq i \leq c$.
\end{proof}

At this point, we have obtained a quite explicit expression for $M^{\rm (T)}_k(z)$.
In the following sections, we will study the 
behaviour near the leading singularities of the quantities above,
for the first two moments, i.e.\ $k=1,2$.
Higher-order moments require a more involved computational machinery that will be presented elsewhere.

\subsubsection{Singularity analysis for $k=1$}

We start our analysis from the simplest case, $k=1$, to illustrate how singularity analysis is applied in this context. 
We expect to recover Equation~\eqref{eq:M1wiener}.
From now on, for simplicity, as the $\barh$ indices are mute summation
indices, we will call them simply $h$.
We have
\begin{equation}
    \begin{split}
        M_1^{\rm (B)}(z) &=
        2 \sqrt{z} B(z)^2 
        \sum_{h_1 \geq 1}
        \left(\log h_1 \right)
        \sqrt{x}^{\,2h_1-1}
        =
        \frac{2(1+x)}{(1-x)^2}
        \sum_{h \geq 1}
        \log h
        \;
        x^h
        =
        \frac{2(1+x)}{(1-x)^2} \Li_{0,1}(x) \, ,
    \end{split}
\end{equation}
and
\begin{equation}
    \begin{split}
        M_1^{\rm (E)}(z) &=
        \sqrt{z} B(z)^2 
        (1-x)^2
        \sum_{h_1 \geq 1}
        \left(\log h_1 \right)
        \sqrt{x}^{\,2h_1-1}
        =
        (1+x) 
        \sum_{h \geq 1}
        \log h
        \;
        x^h
        = (1+x) \Li_{0,1}(x)
        \, ,
    \end{split}
\end{equation}
where $\Li_{s,r}(x) = \sum_{h\geq 1} h^{-s} \left( \log(h) \right)^r
x^h$ is the generalized polylogarithm function \cite{fill2005SingularityAnalysisHadamard}. 

In both cases, the dominant singularity is at $x(z)=1$, i.e.\ at $z=1/4$.
We have
\begin{equation}
    \begin{split}
x\left( z \right) = 1 - 2\sqrt{1-4z} + 2(1-4z) + \bigO \left(
(1-4z)^{\frac{3}{2}} \right) 
\quad \text{for } z \rightarrow \left( \smfrac{1}{4} \right)^- \, ,
    \end{split}
\end{equation}
and 
\begin{equation}
    \begin{split}
        \Li_{0,1}(x) = \frac{\LOG(x) - \gamma_E}{1-x} + \bigO\left(\LOG(x)\right) 
        \quad \text{for } x \rightarrow 1^-
        \, 
    \end{split}
\end{equation}
where $\LOG(x) = \log\left( (1-x)^{-1} \right)$. Here and in the
following, the rewriting of
$\Li_{\alpha,r}(x)$ (for $-\alpha, r \in \mathbb{N}$) in the form
$P(\LOG(x))/(1-x)^{1-\alpha}$, with $P(y)$ a polynomial of degree $r$,
can be done either by matching the asymptotics of the coefficients in the
two expressions (and appealing to the Transfer Theorem), or by using
the explicit formulas in
\cite[Thm.\ VI.7]{flajolet2009AnalyticCombinatorics}.
In this paper we mostly adopt the first strategy.
Passing to the variable $z$ gives
\begin{equation}
    \begin{split}
        \Li_{0,1}\left( x(z)  \right)  
        = \frac{\LOG(4z) - 2 \gamma_E - 2\log2}{4 \sqrt{1-4z}} +
        \bigO\left( \LOG(4z) \right)  
\quad \text{for } z \rightarrow \left( \smfrac{1}{4}\right)^-
        \, .
    \end{split}
\end{equation}
Thus, the singular expansion of $M^{\rm (T)}_1(z)$ is given by
\begin{equation}\label{eq:M1BZasy}
    \begin{split}
        M_1^{\rm (B)}(z) 
        = \frac{\LOG(4z) - 2 \gamma_E - 2\log2}{4
          (1-4z)^{\frac{3}{2}}} + \bigO\left( \frac{\LOG(4z)}{1-4z}
        \right)  
\quad \text{for } z \rightarrow \left( \smfrac{1}{4}\right)^-
        \,
    \end{split}
\end{equation}
and
\begin{equation}\label{eq:M1EZasy}
    \begin{split}
        M_1^{\rm (E)}(z) 
        = \frac{\LOG(4z) - 2 \gamma_E - 2\log2}{2 \sqrt{1-4z}} +
        \bigO\left( \LOG(4z) \right)  
\quad \text{for } z \rightarrow \left( \smfrac{1}{4}\right)^-
        \, .
    \end{split}
\end{equation}
The behaviour of $T_N M^{\rm (T)}_{N,1}$ for large $N$ can be now
estimated by using the so-called \emph{transfer theorem} (see
\cite{flajolet2009AnalyticCombinatorics}, in particular Chapter VI for
general informations, and the table in Figure VI.5 for the explicit
formulas), that allows to jump back and forth between singular
expansion of generating functions and the asymptotic expansion at
large order of their coeffcients.  In practice, we can expand the
approximate generating functions given in Equations~\eqref{eq:M1BZasy}
and~\eqref{eq:M1EZasy} to get an asymptotic approximation for $T_N
M^{\rm (T)}_{N,1}$.  Recalling that
\begin{align}\label{eq:TNasy}
    B_N &= \frac{4^N}{\sqrt{\pi}N^{\frac{1}{2}}} \left( 1+\bigO(N^{-1}) \right) \ef , &
    C_N &= \frac{4^N}{\sqrt{\pi}N^{\frac{3}{2}}} \left( 1+\bigO(N^{-1}) \right)  \, 
\end{align}
for $N\rightarrow\infty$,
we obtain an asymptotic expansion for the first moment of $S(\vsig)$
(which agrees with what we already found in Equation~\eqref{eq:M1wiener})
\begin{equation}
    \begin{split}
        M^{\rm (B)}_{1,N} &= \frac{1}{2} N \log N - \frac{\gamma_E+2}{2} N + \bigO\left(\sqrt{N} \log(N)\right) \\ 
        M^{\rm (E)}_{1,N} &=  \frac{1}{2} N \log N - \frac{\gamma_E}{2} N + \bigO\left(\sqrt{N} \log(N)\right) \, .
    \end{split}
\end{equation}
Notice that, although we have truncated our perturbative series at the
first significant order, in principle the combinatorial method gives
us access to finite-size corrections at arbitrary finite order. 

\subsubsection{Singularity analysis for $k=2$}
\label{sec:mom2}

For $k=2$, we compute Equation~\eqref{eq:Mkz2} by studying separately terms at different values of $c$.
Let us start from bridges.
For $c=1$, and thus $\nu_1=2$, we have
\begin{equation}
    \begin{split}
        M^{\rm (B)}_2(z) \vert_{c=1} 
                =4 \sqrt{z} B(z)^2 
                \sum_{h_1 \geq 1}
                \frac{\left(\log h_1 \right)^2}{2}
                \sqrt{x}^{\,2h_1-1}
                =
                \frac{2(1+x)}{(1-x)^2} \Li_{0,2}(x)
        \, 
    \end{split}
\end{equation}
while for $c=2$, and thus $\nu_1=\nu_2=1$, we have
\begin{equation}
    \begin{split}
        M^{\rm (B)}_2(z) \vert_{c=2} 
                =4 z B(z)^3
                \sum_{h_1,h_2 \geq 1}
                \log h_1 
                \log h_2
                \sqrt{x}^{\,h_1}
                \left( \sqrt{x}^{\,|h_2-h_1+1|} + \sqrt{x}^{\,h_2+h_1-1} \right) 
                \sqrt{x}^{\,h_2-1}
        \, .
    \end{split}
\end{equation}
The presence of the absolute value $|h_2-h_1+1|$ forces us to consider separately the case $h_1 > h_2$ and $h_1 \leq h_2$.
In the first case we get
\begin{equation}
    \begin{split}
           &
           4 z B(z)^3
           \sum_{h_1 > h_2 \geq 1}
           \left(\log h_1 \log h_2 \right)
           \sqrt{x}^{\,h_1}
           (\sqrt{x}^{\,h_1-h_2-1}+\sqrt{x}^{\,h_1+h_2-1})
           \sqrt{x}^{\,h_2-1}
           \\
           &=
           4 z B(z)^3 
           \bigg(
               \sum_{h_1 \geq 1}
               \left(\log (h_1+1) \log (h_1!) \right)
               x^{h_1}
               +
               \sum_{h_1 > h_2 \geq 1}
               \left(\log h_1 \log h_2 \right)
               x^{h_1+h_2-1}
           \bigg)
           \ef,
    \end{split}
\end{equation}
while in the second case we obtain
\begin{equation}
    \begin{split}
          &
          4 z B(z)^3 
          \sum_{1 \leq h_1 \leq h_2}
          \left(\log h_1 \log h_2 \right)
          \sqrt{x}^{\,h_1}
          (\sqrt{x}^{\,h_2-h_1+1}+\sqrt{x}^{\,h_1+h_2-1})
          \sqrt{x}^{\,h_2-1}
          \\
          &=
          4 z B(z)^3 
          \bigg(
              \sum_{h_2 \geq 1}
              \left(\log h_2 \log h_2! \right)
              x^{h_2}
              +
              \sum_{1 \leq h_1 \leq h_2}
              \left(\log h_1 \log h_2 \right)
              x^{h_1+h_2-1}
          \bigg)
          \ef.
    \end{split}
\end{equation}
The combination of these two terms gives
\begin{equation}
    \begin{split}
        M_2^{\rm (B)}(z)|_{c=2}
           &=
           \frac{4 x (1+x)}{(1-x)^3}
           \bigg(
               \sum_{h \geq 1}
               \left(\log (h^2+h) \log h! \right)
               x^{h}
               +
               \frac{1}{x} (\Li_{0,1} )^2
           \bigg)
           \ef .
    \end{split}
\end{equation}
In the case of excursions, the computations are completely analogous, and give
\begin{equation}
    \begin{split}
        M^{\rm (E)}_2(z) \vert_{c=1} 
        =2 \sqrt{z} B(z)^2 (1-x)^2
                \sum_{h_1 \geq 1}
                \frac{\left(\log h_1 \right)^2}{2}
                \sqrt{x}^{\,2h_1-1}
                =
                (1+x) \Li_{0,2}(x)
        \, 
    \end{split}
\end{equation}
\begin{equation}
    \begin{split}
        M_2^{\rm (E)}(z)|_{c=2}
           &=
           \frac{2x(1+x)}{1-x}
           \bigg(
               \sum_{h \geq 1}
               \left(\log (h^2+h) \log h! \right)
               x^{h}
               - (\Li_{0,1} )^2
           \bigg)
           \ef .
    \end{split}
\end{equation}
In order to compute the singular expansion of $\Li_{0,2}(x)$ and of
$\sum_{h\geq1}( \log(h^2+h) \log(h!) x^h)$, one can again use the
transfer theorem, obtaining
\begin{equation}
    \begin{split}
        \Li_{0,2}(x) = \frac{\LOG(x)^2 - 2\gamma_E \LOG(x) + \gamma_E^2 + \frac{\pi^2}{6}}{1-x} + \bigO\left( \LOG(x)^2 \right) 
         \, 
    \end{split}
\end{equation}
and 
\begin{equation}
    \begin{split}
        &\sum_{h\geq1}( \log(h^2+h) \log(h!) z^h) = \\
        &\quad
        \frac{2 \LOG(x)^2 +2(1-2\gamma_E) \LOG(x) + \frac{\pi^2}{3} +2\gamma_E^2-2\gamma_E-2   }{(1-x)^2}
        + \bigO\left( \frac{\LOG(x)^2}{1-x} \right) 
         \, .
    \end{split}
\end{equation}
We see that, for both Dyck bridges and paths, the $c=1$ term is
subleading with respect to the $c=2$ term, by a factor $1-x$ which,
after use of the Transfer Theorem, implies a factor
$\mathcal{O}(N^{-1})$. It is easy to imagine (and in agreement with
the discussion in Section~\ref{sec:integral}) that this pattern will
hold also for all subsequent moments, that is, the leading term for
the $k$-th moment will be given by the $c=k$ contribution alone, all
other terms altogether giving a correction $\mathcal{O}(N^{-1})$.

After substituting $x(z)$ with its expansion around $z = 1/4$, and
after having performed a series of tedious but trivial computations,
we obtain
\begin{equation}\label{eq:mom2com}
    \begin{split}
        M^{\rm (B)}_{2,N} &= \frac{1}{4} N^2 \left( \log N \right)^2 - \frac{\gamma_E+2}{2} N^2 \log N + \left( \frac{4}{3} + \frac{\gamma_E^2}{4} + \gamma_E - \frac{\pi^2}{72} \right)  N^2 + \bigO\left (N^\frac{3}{2} \left( \log N \right)^2\right) \, ,\\ 
        M^{\rm (E)}_{2,N} &= \frac{1}{4} N^2 \left( \log N \right)^2 - \frac{\gamma_E}{2} N^2 \log N + \left( \frac{\gamma_E^2}{4}+\frac{5 \pi ^2}{24}-2 \right)  N^2 + \bigO\left (N^\frac{3}{2} \left( \log N \right)^2\right) \, .
    \end{split}
\end{equation}
Finally, we recover the moments of $s(\vsig)$
\begin{equation}
    \begin{split}
        \langle (s(\vsig))^2 \rangle_{\rm B} 
        &= \frac{4}{3} + \frac{\gamma_E^2}{4} + \gamma_E - \frac{\pi^2}{72} + \bigO\left( \frac{(\log N)^2}{\sqrt{N}} \right) \, ,\\
        \langle (s(\vsig))^2 \rangle_{\rm E} &=
        \frac{\gamma_E ^2}{4}+\frac{5 \pi ^2}{24} -2
        + \bigO\left( \frac{(\log N)^2}{\sqrt{N}} \right) 
        \, .
    \end{split}
\end{equation}
Details for these computations can be found in Appendix~\ref{app:mom2}.

\section{Numerical results}
\label{sec:numerics}

To check our analytical predictions, we performed an exact sampling of
our configurations, and collected a statistics on the resulting
entropies, from which 
we estimated the distribution of the rescaled entropy
\begin{equation}
    \begin{split}
        s(\vsig) = \frac{1}{N}\left( S(\vsig) - \frac{1}{2}\log N \right)  \, .
    \end{split}
\end{equation}
For both bridges and excursions, and
for values of $N=10^3,10^4,10^5,10^6$, we sampled uniformly $10^5$
random paths.

Figure~\ref{fig:numerics} summarises our results.  We clearly see that
as $N$ grows larger and larger, the prediction for the first two
moments of $s$ matches better and better the empirical value in both
ensembles.  The distribution of $s$ is clearly non-Gaussian in the
case of Dyck bridges.  For Dyck excursion, a quick Kolmogorov test
rules out the Gaussian hypothesis (in particular, more easily, the
4-th centered moment is only $\sim 2.7$ times the 2-nd centered moment
squared, instead of a factor $3$ required for gaussianity). However,
at the present statistical precision we cannot rule out the hypothesis
that the centered distribution is symmetric, i.e.\ that all the
centered odd moments vanish, although we have no theoretical argument
for conjecturing this fact, neither from the probabilistic approach
of Section \ref{sec:integral}, nor from the combinatorial approach of 
Section \ref{sec:combinatorial}. It would be interesting to understand
the reasons of this unexpected numerical finding.

The code and the raw data used to produce Figure~\ref{fig:numerics}
are available at
\url{https://github.com/vittorioerba/EntropyMatching}.

\begin{figure}
    \centering
    \begin{tikzpicture}
        \begin{scope}
            \node at (3.5,-3.5) {\includegraphics[height=3cm]{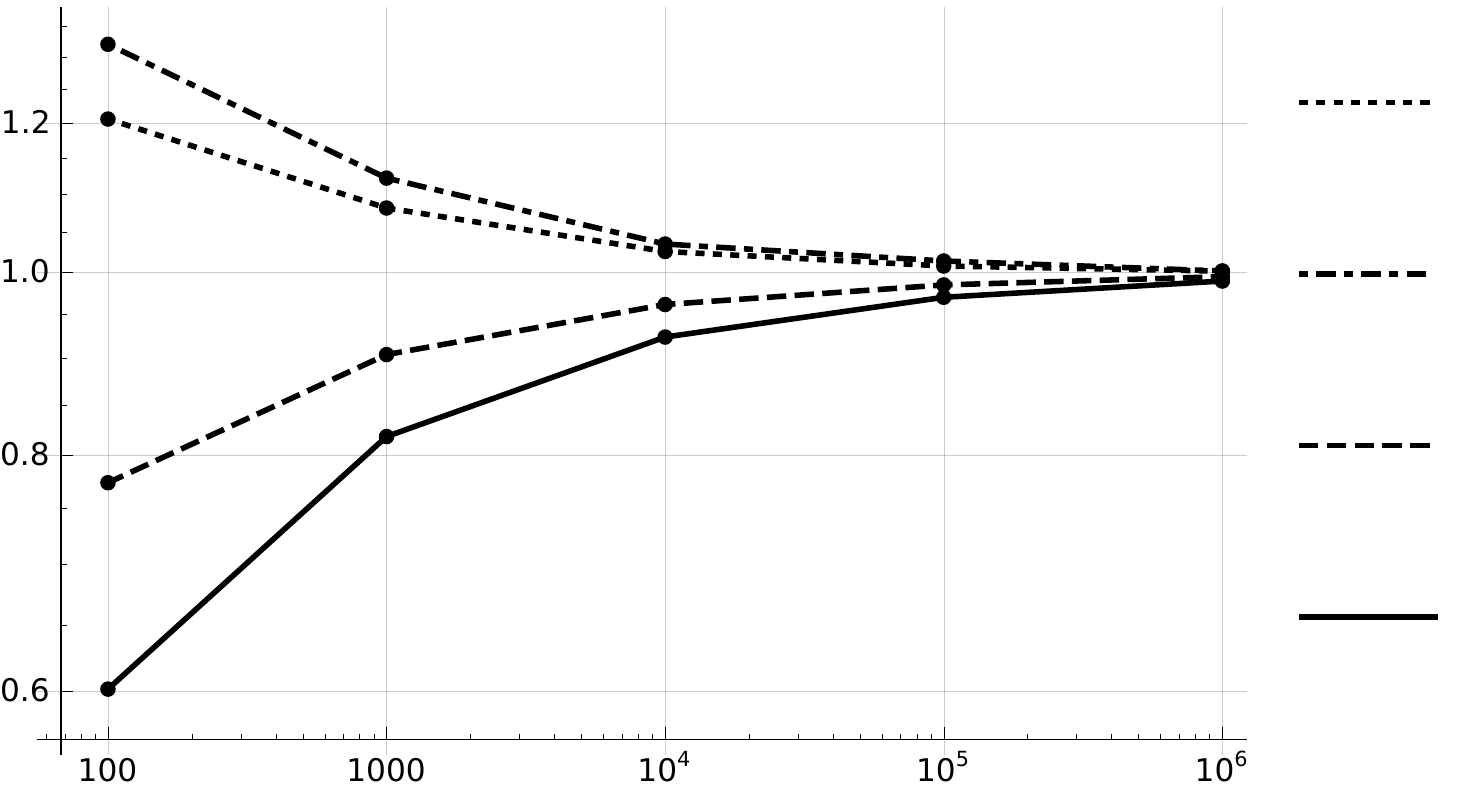}};
            \node at (0,0) {\includegraphics[height=3cm]{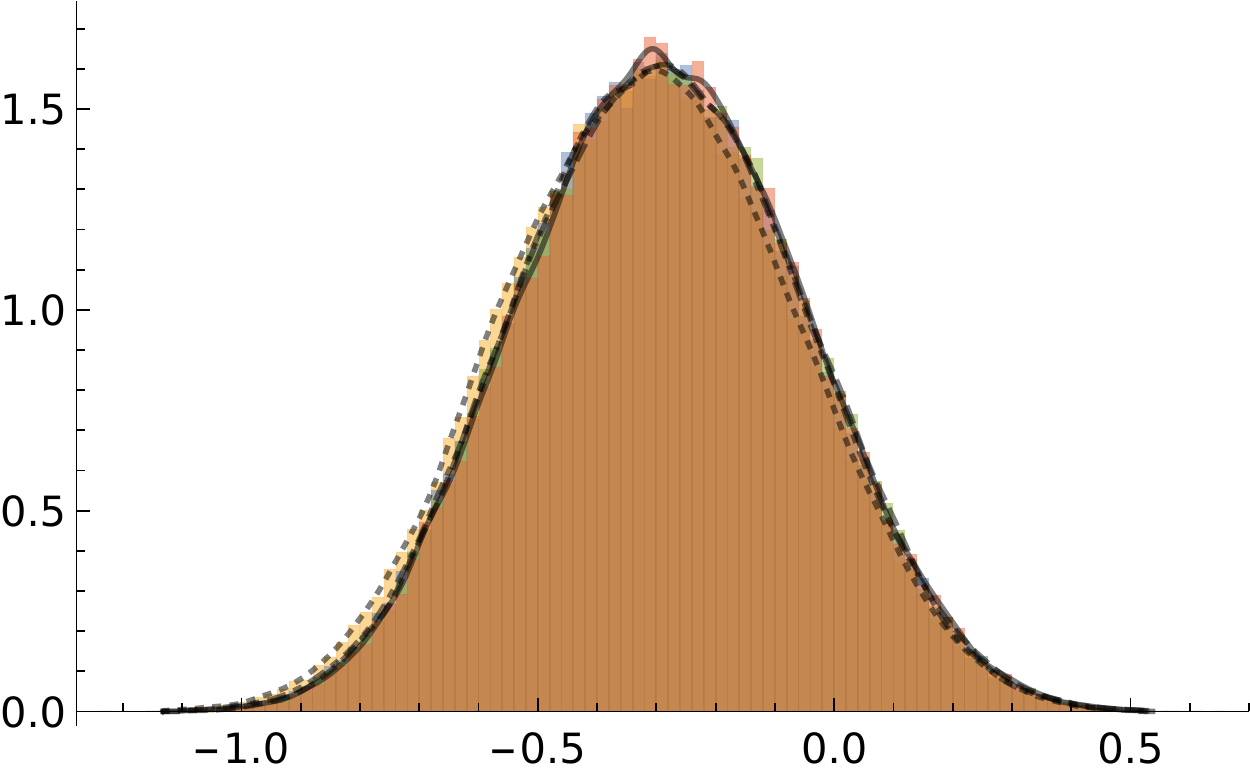}};
            \node at (7,0) {\includegraphics[height=3cm]{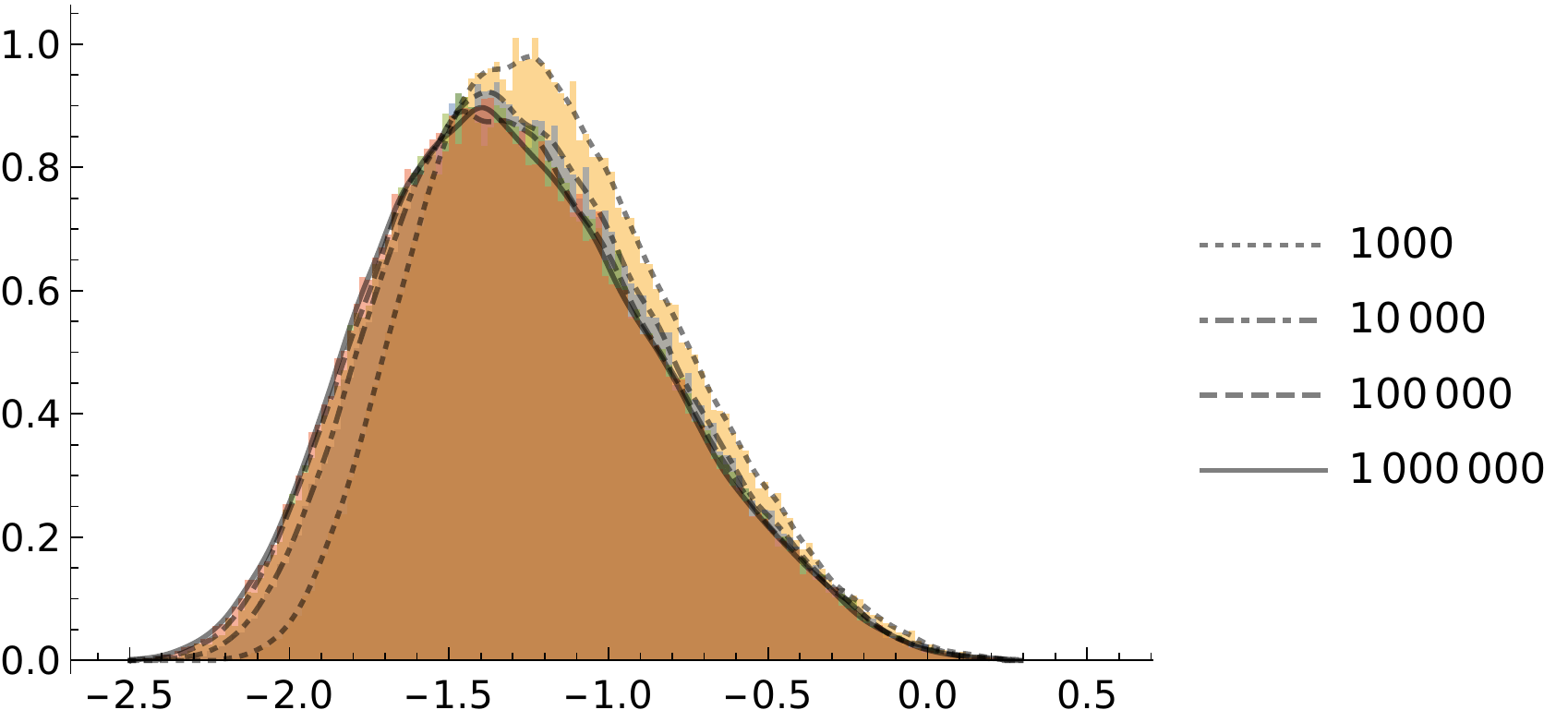}};
            \node at (0,2) {Dyck excursions};
            \node at (6,2) {Dyck bridges};
            \node at (9.2,1) {\footnotesize Size $N$};
            \node at (6,-5) {\footnotesize Size $N$};
            \node[rotate=90, text width=2.5cm, align=center] at (-0.25,-3.4) {\footnotesize Ratio between empirical and predicted value};
            \node at (6.75,-2.3) {\footnotesize $M^{\rm (E)}_{1,N}$};
            \node at (6.75,-3) {\footnotesize $M^{\rm (E)}_{2,N}$};
            \node at (6.75,-3.6) {\footnotesize $M^{\rm (B)}_{1,N}$};
            \node at (6.75,-4.3) {\footnotesize $M^{\rm (B)}_{2,N}$};
        \end{scope}
    \end{tikzpicture}
    \caption{\footnotesize 
        \textbf{Top:} distribution of $s$ for the Dyck excursions (left) and Dyck bridges (right) ensembles at various sizes $N$. The shaded fillings are histograms, while the black profiles are kernel density estimates.
        \textbf{Bottom:} ratio between empirical and predicted values of the first two moments as a function of the size $N$, for both ensembles.
    }
    \label{fig:numerics}
\end{figure}
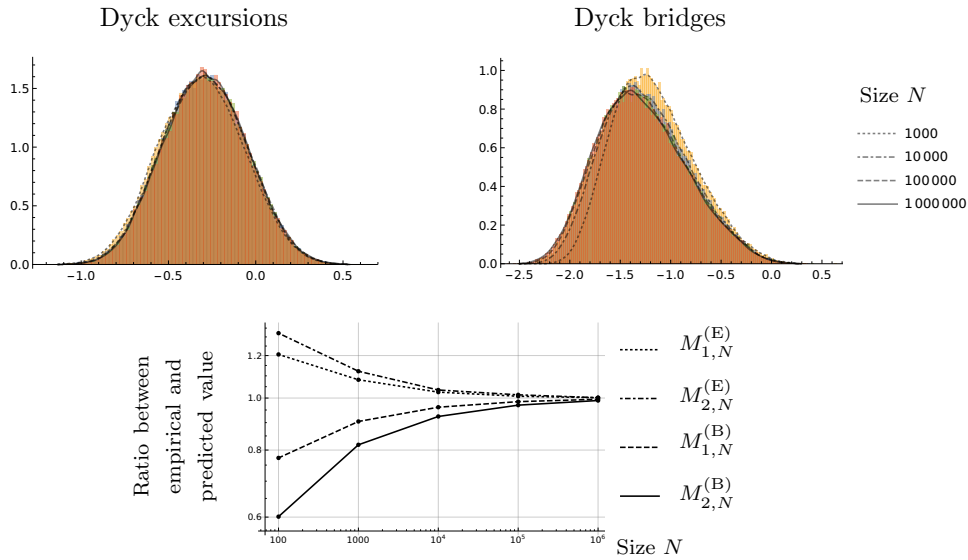

\section{Acknowledgements}

A. Sportiello is partially supported by the Agence Nationale de la
Recherche, Grant Numbers ANR-18-CE40-0033 (ANR DIMERS) and
ANR-15-CE40-0014 (ANR MetACOnc).

\appendix

\section{The first two moments of the rescaled entropy in the integral representation}
\label{app:momint}

In this Appendix we provide the computation for the first two moments of the rescaled entropy in the bridges ensemble, using the integral representation.
The case of excursion can be treated analogously. 

\subsection{Useful formulas}

We start by providing some useful identities.

We start giving an explicit representation for the non-integer Gaussian moments:
\begin{equation}
  \int_{- \infty}^{\infty} d y \, y^{2 k} e^{- \frac{y^2}{\lambda}} = \Gamma
  \left( k + \frac{1}{2} \right) \lambda^{k + \frac{1}{2}} . \label{Gaussian}
\end{equation}

In the following we shall also use the duplication formula for the Gamma function
\begin{equation}
    \begin{split}
   \Gamma (s) \Gamma \left( s + \frac{1}{2} \right) = 2^{1 - 2 s} \sqrt{\pi}
   \Gamma (2 s)       \, , 
    \end{split}
\end{equation}
and the expansion for the hyperbolic cosine
\begin{equation}
    \begin{split}
   \cosh (2 z) = \sum_{s \geqslant 0} \frac{(2 z)^{2 s}}{(2 s) !} = \sqrt{\pi}
   \sum_{s \geqslant 0} \frac{1}{\Gamma \left( s + \frac{1}{2} \right)} 
   \frac{z^{2 s}}{s!}      \, .
    \end{split}
\end{equation}

Finally, we recall the definition of the hypergeometric function
\begin{equation}
    \begin{split}
   {}_2 F_1 (a, b ; c ; z) = \frac{\Gamma (c)}{\Gamma (a) \Gamma (b) } 
   \sum_{s \geqslant 0} \frac{\Gamma (s + a) \Gamma (s + b)}{\Gamma (s + c) } 
   \frac{z^s}{s!}       \, ,
    \end{split}
\end{equation}
and the Euler identity
\begin{equation}
    \begin{split}
   {}_2 F_1 (a, b ; c ; z) = (1 - z)^{c - b - a}  {}_2 F_1 (c - a, c - b ; c ;
   z)       \, 
    \end{split}
\end{equation}
which implies that
\begin{equation}
    \begin{split}
  \sum_{s \geqslant 0} \frac{\Gamma (s + a) \Gamma (s + b)}{\Gamma (s + c) } 
  \frac{z^s}{s!} 
  & =  (1 - z)^{c - b - a}  \frac{\Gamma (a) \Gamma
  (b)}{\Gamma (c - a) \Gamma (c - b)}  \sum_{s \geqslant 0} \frac{\Gamma (s + c - a) \Gamma (s + c -
  b)}{\Gamma (s + c) }  \frac{z^s}{s!} \label{ide} 
         \, .
    \end{split}
\end{equation}

\subsection{First moment}
\label{app:mom1integral}

We wish to compute
\begin{equation}
    \begin{split}
   \langle s [\vsig] \rangle_{\rm B}  = \int_0^1 d t \int_{-
   \infty^{}}^{\infty} d x \frac{\log x^2}{2}  \frac{e^{- \frac{x^2}{4 t (1 -
   t) }}}{\sqrt{4 \pi t_{} (1 - t_{})}} = \frac{1}{2} \int_{0}^1 dt I_1(t)       \, .
    \end{split}
\end{equation}
First of all, we substitute $\log x^2 \to x^{2k} $. We will later take the derivative in $k$ and evaluate our expressions for $k=0$ to obtain back the logarithmic contribution.

Thus, we start from
\begin{equation}
    \begin{split}
   I_1 (t,k) 
   = \int_{- \infty^{}}^{\infty} d x x^{2 k_{}} \frac{e^{-
   \frac{x^2}{4 t (1 - t) }}}{\sqrt{4 \pi t_{} (1 - t_{})}} 
   = \frac{\lambda^{k
   + \frac{1}{2}} \Gamma \left( k + \frac{1}{2} \right)}{\sqrt{\pi \lambda}}
   = \frac{\lambda^{k} \Gamma \left( k + \frac{1}{2} \right)}{\sqrt{\pi}}
   \, 
    \end{split}
\end{equation}
where $\lambda = 4 t (1 - t)$
so that
\begin{equation}
    \begin{split}
   I_1 (t,k) = \frac{[4 t_{} (1 - t)^{}]^k}{\sqrt{\pi}} \Gamma \left( k +
   \frac{1}{2} \right)       \, 
    \end{split}
\end{equation}
Then, 
\begin{equation}
    \begin{split}
        I_1 (t) = \partial_k I_1 (t,k) \vert_{k=0}   
            = 
   \left[ \log [4 t_{} (1 - t)^{}] + \psi_0 \left( \frac{1}{2} \right) \right]
        \, 
    \end{split}
\end{equation}
where $\psi_0 (z) = \frac{d}{d z} \log \Gamma (z) $ is the digamma function,
and
\begin{equation}
    \begin{split}
\psi_0 \left( \frac{1}{2} \right) = - \gamma_E - \log 4 \, .
    \end{split}
\end{equation}

Finally, 
\begin{equation}
    \begin{split}
   \langle s [\vsig] \rangle_{\rm B}  = \frac{1}{2} \int_0^1 d t \{ \log
   [t_{} (1 - t)^{}] - \gamma_E \} = - \frac{2 + \gamma_E}{2}        \, .
    \end{split}
\end{equation}

\subsection{Second moment}
\label{app:mom2integral}

We wish to compute
\begin{equation}\label{eq:mom2int1}
    \begin{split}
        \left \langle (s[\vsig])^2 \right\rangle_{\rm B}  
        &= 2! \int_{\Delta_2 }  dt_1 dt_2 \frac{1}{4 \pi \sqrt{ t_1  (t_2 - t_1 )  (1-t_2 )}} \\
        &\qquad\times \int_{-\infty}^{\infty} dx dy \frac{\log x^2 }{2} \frac{\log y^2 }{2} \exp \left[ -\frac{x^2}{4t_1 } - \frac{(x-y)^2}{4(t_2 -t_1 )}  - \frac{x^2}{4(1-t_2 )} \right] 
        \\
        &= \frac{1}{2} \int_{\Delta_2 }  dt_1 dt_2  \, I_2 (t_1 ,t_2 )
        \,.
    \end{split}
\end{equation}
Again, we substitute $\log x^2 \to x^{2k_1}$ and $\log y^2 \to y^{2k_2}$, and we will recover the correct logarithmic factors by taking derivatives in $k_1$ and $k_2$ later.   
We start by dealing with the contact term:
\begin{equation}
    \begin{split}
        \exp \left[ -\frac{(x-y)^2}{4 (t_2 -t_1 )} \right]
        &=
        \exp \left[ -\frac{x^2 + y^2}{4 (t_2 -t_1 )} \right] \cosh \left( \frac{x y}{2 (t_2 - t_1 )} \right)
        \\
        &=
        \exp \left[ -\frac{x^2 + y^2}{4 (t_2 -t_1 )} \right] \sum_{s \geq 0} \frac{\sqrt{\pi}}{\Gamma\left(s+\frac{1}{2}\right) s!} \left( \frac{xy}{4(t_2 - t_1)} \right)^{2s} 
        \, 
    \end{split}
\end{equation}
where the hyperbolic sine was discarded due to parity in $x$ and $y$.
With this substitution, the integrals in $x$ and $y$ are decoupled, and can be evaluated as
\begin{equation}
    \begin{split}
        I_2 &(t_1 ,t_2 ,k_1 ,k_2 )  = \\
        &= 
        \frac{1}{4 \pi \sqrt{t_1 (t_2 -t_1 )(1-t_2 )}}
        \int_{-\infty}^{\infty} dx dy \,  x^{2k_1} y^{2k_2} \exp \left[ -\frac{x^2}{4t_1 } - \frac{(x-y)^2}{4(t_2 -t_1 )}  - \frac{x^2}{4(1-t_2 )} \right] \\
        &= \frac{2}{\pi \sqrt{\Delta_1 \Delta_2 \Delta_3 }}
        \sum_{s \geq 0} \frac{\sqrt{\pi}}{\Gamma\left(s+\frac{1}{2}\right) s! \Delta_2 ^{2s} }  
            \int_{-\infty}^{\infty} dx x^{2(s+k_1 ) } \exp\left[ - \frac{x^2}{\lambda_1 }\right]
            \int_{-\infty}^{\infty} dy y^{2(s+k_2 )} \exp\left[ - \frac{y^2}{\lambda_2 }\right] 
            \\
        &= \frac{2 \lambda_1 ^{k_1 +\frac{1}{2} } \lambda_2 ^{k_2 +\frac{1}{2} }}{\sqrt{\pi \Delta_1 \Delta_2 \Delta_3 }}
    \sum_{s \geq 0} \frac{ \Gamma\left( k_1 +s+\frac{1}{2} \right) \Gamma\left(k_2 + s+ \frac{1}{2}\right)}{\Gamma\left(s+\frac{1}{2}\right) s!} \left( \frac{ \lambda_1 \lambda_2 }{\Delta_2 ^2}\right)^s
        \, 
    \end{split}
\end{equation}
where 
\begin{align}
    \lambda_1 &= \frac{4 t_1 (t_2 - t_1 ) }{t_2 } \, , & \lambda_2 &= \frac{4 (t_2 -t_1 )(1 - t_2 )}{1-t_1 } \, , &  \Delta_i &= 4(t_{i}  -t_{i-1} ) \, ,
\end{align}
(we recall that $t_0 =0$ and $t_3 = 1$ by convention).
We also notice that the following identity holds
\begin{equation}
    \begin{split}
        \frac{2 \sqrt{\lambda_1 \lambda_2}}{\sqrt{\Delta_1 \Delta_2 \Delta_3}} \left(1- \frac{\lambda_1 \lambda_2}{\Delta_2 ^2}\right)^{-\frac{1}{2}} = 1
         \, 
    \end{split}
\end{equation}
as it can be explicitly verified by substituting the definitions.

By using the previous identity and Equation~\eqref{ide}, we can rewrite $I_2 (t_1 ,t_2 ,k_1 ,k_2 )$  as
\begin{equation}
    \begin{split}
        I_2 &(t_1 ,t_2 ,k_1 ,k_2 ) =
        \\
            &= \frac{\lambda_1 ^{k_1 } \lambda_2 ^{k_2 }}{ \sqrt{\pi}}
        \left(1-\frac{\lambda_1 \lambda_2 }{\Delta_2 ^2}\right)^{-k_1 -k_2 }
        \frac{\Gamma\left( k_1 +\frac{1}{2} \right) \Gamma\left( k_2 +\frac{1}{2}\right)}{\Gamma\left(-k_1\right) \Gamma\left(-k_2 \right)}
    \sum_{s \geq 0} \frac{ \Gamma\left( s - k_1 \right) \Gamma\left(s- k_2 \right)}{\Gamma\left(s+\frac{1}{2}\right) s!} \left( \frac{ \lambda_1 \lambda_2 }{\Delta_2 ^2}\right)^s
    \\
            &= k_1 k_2 \left[ \sqrt{\pi} \sum_{s\geq 1} \frac{\Gamma\left(s\right) \Gamma\left(s\right)}{\Gamma\left(s+\frac{1}{2}\right) s!} \left(\frac{\lambda_1 \lambda_2 }{\Delta_2 ^2}\right)^s + \bigO(k_1 ) \bigO(k_ 2)  \right] 
            \\
            &\quad
            +
\frac{\lambda_1 ^{k_1 } \lambda_2 ^{k_2 }}{\pi}
        \left(1-\frac{\lambda_1 \lambda_2 }{\Delta_2 ^2}\right)^{-k_1 -k_2 }
        \Gamma\left( k_1 +\frac{1}{2} \right) \Gamma\left( k_2 +\frac{1}{2}\right)
    \, ,
    \end{split}
\end{equation}
where in the last line we used the fact that $\Gamma(-k) = k^{-1} + \dots$ as $k$ goes to zero to highlight the linear dependence in $k_1$ and $k_2$ in the first term. 
Notice that
\begin{equation}
    \begin{split}
         \sum_{s\geq 1} \frac{\Gamma\left(s\right) \Gamma\left(s\right)}{\Gamma\left(s+\frac{1}{2}\right) s!} \left(\frac{\lambda_1 \lambda_2 }{\Delta_2 ^2}\right)^s  
         = \frac{2}{\sqrt{\pi}} \arcsin^2 \left(\sqrt{z}\right)
 \,.
    \end{split}
\end{equation}

We can now take the derivative with respect to $k_1$ and $k_2$ and evaluate the expression in $k_1 =k_2 =0$ to obtain $I_2 (t_1 ,t_2 )$
\begin{equation}
    \begin{split}
        I_2 (t_1 ,t_2 )
        &= \partial ^2 _{k_1 ,k_2 } I_2 (t_1 ,t_2 ,k_1 ,k_2 ) \vert_{k_1 =k_2 = 0 }  
        \\
        &=  2 \arcsin^2 \left(\sqrt{\frac{\lambda_1 \lambda_2 }{\Delta_2 ^2}}\right)
        \\
        &\quad
        + 
        \left[ \psi_0 \left(\frac{1}{2}\right) + \log \lambda_1 - \log\left(1 - \frac{\lambda_1 \lambda_2 }{\Delta_2 ^2}\right) \right]
        \left[ \psi_0 \left(\frac{1}{2}\right) + \log \lambda_2 - \log\left(1 - \frac{\lambda_1 \lambda_2 }{\Delta_2 ^2}\right) \right]
        \\
        &= 2 \arcsin^2 \left(\sqrt{\frac{ t_1 (1-t_2 )}{(1-t_1 )t_2 }}\right)
        \\
        &\quad
        +  
        \left[ \psi_0 \left(\frac{1}{2}\right) + \log \left(4 t_1 (1-t_1 )\right) \right]
        \left[ \psi_0 \left(\frac{1}{2}\right) + \log \left(  4 t_2 (1-t_2 )\right) \right]
        \, .
    \end{split}
\end{equation}

The second line is easily treated by noting that it is symmetric under $t_1 \to 1- t_1$, giving
\begin{equation}
    \begin{split}
        &
        \frac{1}{2} \int_0 ^1 dt_1 \int_{t_1} ^1 dt_2   
        \left[ \psi_0 \left(\frac{1}{2}\right) + \log \left(4 t_1 (1-t_1 )\right) \right]
        \left[ \psi_0 \left(\frac{1}{2}\right) + \log \left(  4 t_2 (1-t_2 )\right) \right]
        \\
        &\quad
        =
        \frac{1}{4} \int_0 ^1 dt_1 \int_0 ^1 dt_2   
        \left[ \psi_0 \left(\frac{1}{2}\right) + \log \left(4 t_1 (1-t_1 )\right) \right]
        \left[ \psi_0 \left(\frac{1}{2}\right) + \log \left(  4 t_2 (1-t_2 )\right) \right]
        \\
        &\quad
        =\left[ \frac{1}{2}\int_0 ^1 dt  
        \left( \log \left(t (1-t )\right) - \gamma_E \right)
        \right]^2
        \\
        &\quad
        = \left( - \frac{\gamma_E +2}{2} \right)^2
        = \left( \left\langle s[\vsig] \right\rangle_{\rm B}  \right)^2
        \, .
    \end{split}
\end{equation}

Thus, the variance of $s[\vsig]$ is given by 
\begin{equation}
    \begin{split}
        \left\langle \left( s[\vsig] \right)^2 \right\rangle_{\rm B} -
        \left( \left\langle s[\vsig] \right\rangle_{\rm B}  \right)^2
        &=
        \frac{1}{2} \int_0 ^1 dt_1 \int_{t_1}^1 dt_2 
        \, 2 \arcsin^2 \left( \sqrt{ \frac{t_1 (1-t_2 )}{(1-t_1 ) t_2}} \right) 
        \\
        &= \int_0 ^1 dt \int_0 ^1 dz  \,
        \frac{t(1-t)}{(t+z-tz)^2}
        \arcsin^2 \left( \sqrt{z} \right) 
        \\
        &= \int_0 ^1 dz \left[  - \frac{2(1-z) + (1+z) \log z}{(1-z)^3} \right]
        \arcsin^2 \left( \sqrt{z} \right) 
        \\
        &= \frac{1}{3} - \frac{\pi^2}{72} \, .
        \, 
    \end{split}
\end{equation}

\section{Computations needed in Section~\ref{sec:mom2}}\label{app:mom2}

In this Appendix, we present in greater detail the computations sketched in Section~\ref{sec:mom2}.

First of all, let us state the transfer theorem of singularity analysis in a slightly unprecise, but operatively correct form. See \cite[Chapter VI]{flajolet2009AnalyticCombinatorics} for the details.
\begin{theorem}
    Let $f(z)$, $g(z)$ and $h(z)$ be generating function with unit
    radius of convergence, and let $f_n, g_n$ and $h_n$ be their
    coefficients. Then
    \begin{equation}
        \begin{split}
            f(z) = g(z) + \bigO(h(z)) \qquad z \rightarrow 1^-
        \end{split}
    \end{equation}
    if and only if
    \begin{equation}
        \begin{split}
            f_n \sim g_n + \bigO(h_n) \qquad n \rightarrow \infty \, .
        \end{split}
    \end{equation}
\end{theorem}
\noindent
In practice, one builds a \emph{standard scale} of functions whose
coefficient expansion is known, expands a generic generating function
over the standard scale (assuming that the scale is rich enough to
admit the generating function of interest in its linear span), and uses the transfer theorem to guarantee
that an asymptotic equivalence at the level of the generating functions implies (and is
implied) by an asymptotic equivalence of their coefficients.  A 
standard scale adopted already in 
\cite[Section VI.8]{flajolet2009AnalyticCombinatorics}, and
which is rich enough for our purposes,
is given by functions of the form
\begin{equation}
    \begin{split}
        \left( \frac{1}{1-x} \right)^\alpha \LOG(x)^\beta  \, 
    \end{split}
\end{equation}
where $\LOG(x) = \log\left( (1-x)^{-1} \right)$.

\subsection{Singular behaviour of $\Li_{0,2}(x)$}

The coefficient of $\Li_{0,2}(x)$ equals $\log^2(h)$.
It is easy to see, using 
\cite[Figure VI.5]{flajolet2009AnalyticCombinatorics}, that the
singular expansion of $\Li_{0,2}(x)$ onto the standard scale must be
made by the terms $(1-x)^{-1} L(x)^2, (1-x)^{-1} L(x), (1-x)^{-1}$ and
higher-order terms.  The coefficient of the expansion can be retrieved
by basic linear algebra:
\begin{itemize}
    \item the coefficient of $(1-x)^{-1} L(x)^2$ must be 1 to reconstruct the $\log(h)^2$ term;
    \item the coefficient of $(1-x)^{-1} L(x)$ must be $-2 \gamma_E$ to cancel the $\log(h)$ term introduced by the first standard scale function;
    \item the coefficient of $(1-x)^{-1}$ must be 
$\zeta(2) + \gamma_E^2 = \pi^2/6 + \gamma_E^2$ 
to cancel the constant terms introduced in the expansion by the previous standard scale functions.
\end{itemize}
The error term can be obtained by observing that all expansions used above are valid up to order $\bigO\left( h^{-1} \log h \right)$.
Thus
\begin{equation}
    \begin{split}
        \Li_{0,2}(x) = \frac{\LOG(x)^2 - 2\gamma_E \LOG(x) + \gamma_E^2 + \frac{\pi^2}{6}}{1-x} + \bigO\left( \LOG(x)^2 \right) \, .
    \end{split}
\end{equation}

\subsection{Singular behaviour of $\sum_{h \geq 1} \left(\log (h^2+h) \log h! \right) x^{h}$}
The procedure is analogous to the computation for the $\Li_{0,2}(x)$. We just need to use Stirling's approximation for $\log h!$ to obtain the expansion of the coefficients in $h$.
We have, for the coefficient of order $k$ of this function,
\begin{equation}
    \begin{split}
& \left(\log (h^2+h) \log h! \right) \\
               &\quad= 
               2 h \log ^2(h)-2 h \log (h)+\log ^2(h)+(1+\log (2)+\log (\pi )) \log (h)-1+
\bigO\left(\frac{\log (h)}{h}\right)
        \, 
    \end{split}
\end{equation}
so that the singular expansion must be made by the terms $(1-x)^{-2} L(x)^2, (1-x)^{-2} L(x), (1-x)^{-2}$, and higher-order ones.
The coefficients can be found using the same strategy adopted for $\Li_{0,2}(x)$, obtaining
\begin{equation}
    \begin{split}
               &\sum_{h \geq 1}
               \left(\log (h^2+h) \log h! \right)
               x^{h} = \\
               &\quad
        \frac{
            2 L(x)^2 + 2(1-2\gamma_E) L(x) +2 \gamma ^2+\frac{\pi ^2}{3} -2-2 \gamma_E
        }{(1-x)^2} + \bigO\left( \frac{L(x)^2}{1-x} \right) \, .
    \end{split}
\end{equation}

\subsection{The change of variable $x(z)$}

First of all, we rewrite the change of variable $x(z) = C(z) -1$ in terms of the singular variables $X = (1-x)^{-1}$ and $Z = (1-4z)^{-1}$.
We have that
\begin{equation}
    \begin{split}
        X(Z) = \frac{1}{2} \left( \sqrt{Z} + 1 \right)  \, 
    \end{split}
\end{equation}
so that
\begin{equation}
    \begin{split}
        (1-x)^{-\alpha} &= X^\alpha = \left( \frac{\sqrt{Z}}{2}
      \right)^\alpha \left( 1 + \frac{\alpha}{\sqrt{Z}} +
      \bigO(Z^{-1})  \right) \, ,  \\
\LOG(x)^k &= \log^k(X)
=
\left(\frac{1}{2} \log Z - \log 2 + \bigO(1/\sqrt{Z})
\right)^k
=
\left(\frac{1}{2} \log Z - \log 2
\right)^k
 + \bigO \left(
\frac{\log^{k-1}(Z)}{\sqrt{Z}} \right)
 \, .
    \end{split}
\end{equation}
These relations are enough to convert singular expansions in $X$ into
singular expansions in $Z$ at the leading algebraic order. The
formulas of this subsection are useful, in principle, also at higher
moments.

\subsection{Explicit expressions for $M^{\rm (T)}_2(z)$}

We report here the explicit expression of $M^{\rm (T)}_2(z)$ both as a
function of $x$ and as a function of $z$:
\begin{equation}\label{eq:ap2}
    \begin{split}
        M^{\rm (E)}_2(z(x)) &= \frac{4 L(x)^2 + 8(1-\gamma_E) L(x) -8-8 \gamma_E+4 \gamma_E^2+\frac{4 \pi ^2}{3} }{(1-x)^3}  + \bigO\left( \frac{L(x)^2}{(1-x)^2} \right) \\
        M^{\rm (E)}_2(z) &=
    \left( \frac{1}{8} L(4z)^2 -\frac{1 - \gamma_E - \log (2)}{2} L(4z)\right)\frac{1}{(1-4z)^\frac{3}{2}} \\
                         &\quad+\frac{\pi ^2-6+3 \gamma_E (\gamma_E-2+\log (4))+(\log (2)-2) \log (8)}{6 (1-4z)^\frac{3}{2}}
                         + \bigO\left( \frac{L(4z)^2}{1-4z} \right) \\
        M^{\rm (B)}_2(z(x)) &= \frac{24 L(x)^2 + 16(1-3\gamma_E) L(x)-16-16 \gamma_E+24 \gamma_E^2+\frac{8 \pi ^2}{3}}{(1-x)^5}  + \bigO\left( \frac{L(x)^2}{(1-x)^4} \right) \\
        M^{\rm (B)}_2(z) &=
    \left( \frac{3}{16} L(4z)^2 -\frac{3 \gamma_E+1- 3 \log (2)}{4} L(4z)\right)\frac{1}{(1-4z)^\frac{5}{2}} \\
                         &\quad+\frac{9 \gamma_E^2+\pi ^2-6+6 \gamma_E (\log (8)-1)+(\log (8)-2) \log (8)}{12 (1-4z)^\frac{5}{2}}
                         + \bigO\left( \frac{L(4z)^2}{(1-4z)^2} \right) 
        \, .
    \end{split}
\end{equation}
These expressions can be easily recalculated also with smaller error
terms, repeating the procedure of the previous paragraphs with more
terms in the Taylor expansions, and provide an evaluation of the
variance of our quantities of interest, at the desired order in $N$.

\end{document}